\documentclass[12pt]{article}
\usepackage{fullpage}

\usepackage{amsmath,amstext,amsthm,amssymb}
\usepackage{enumitem}
\usepackage{pdflscape}

\newcommand{\mybox}[1]{\begin{tabular}{c}#1\end{tabular}}
\newcommand{\red}{$^\dagger$}
\newcommand{\odd}{$^1$}
\newcommand{\even}{$^0$}
\newcommand{\unip}{$^u$}

\usepackage{array,color,colortbl} 
\providecommand{\mk}{\cellcolor[gray]{.8}}

\renewcommand{\ge}{\geqslant}
\renewcommand{\le}{\leqslant}

\newcommand{\ord}{\mathop{\mathrm{ord}}}
\newcommand{\sym}{\mathcal{S}}

\def\eref#1{$(\ref{#1})$}
\def\sref#1{\S$\ref{#1}$}
\def\lref#1{Lemma~$\ref{#1}$}
\def\tref#1{Theorem~$\ref{#1}$}
\def\Tref#1{Table~$\ref{#1}$}
\def\iref#1{{\ref{#1}}}

\def\({\bigl(} \def\){\bigr)}
\def\card#1{\mathopen|#1\mathclose|}
\def\Card#1{\bigl|#1\bigr|}

\def\Z{\mathbb{Z}}
\def\id{\varepsilon}

\def\narrowcols{ \addtolength{\arraycolsep}{-1.5pt}\begin{footnotesize} }
\def\normalcols{ \end{footnotesize}\addtolength{\arraycolsep}{1.5pt} }
\def\verynarrowcols{ \addtolength{\arraycolsep}{-3pt}\renewcommand{\arraystretch}{0.8}\begin{scriptsize} }
\def\verynormalcols{ \end{scriptsize}\addtolength{\arraycolsep}{3pt}\renewcommand{\arraystretch}{1.25}}

\def\tspacer{{\vrule height 2.25ex width 0ex depth0ex}}

\def\tdot{{\cdot}}

\newtheorem{theorem}{Theorem}[section]
\newtheorem{lemma}{Lemma}[section]

\begin{document}

\title{Enumeration of Latin squares with conjugate symmetry}

\author{Brendan D. McKay\thanks{Research supported by ARC grant DP190100977.}\\
 \small School of Computing\\[-0.5ex]
  \small Australian National University\\[-0.5ex]
  \small ACT 2601, Australia\\
  \small\tt brendan.mckay@anu.edu.au
  \and
Ian M. Wanless\thanks{Research supported by ARC grant DP150100506.}\\
  \small School of Mathematics\\[-0.5ex]
  \small Monash University\\[-0.5ex]
  \small Vic 3800, Australia\\
  \small\tt ian.wanless@monash.edu
}

\date{}

\maketitle

\begin{abstract}
A Latin square has six conjugate Latin squares obtained by
uniformly permuting its (row, column, symbol) triples. We say that a Latin
square has conjugate symmetry if at least two of its six conjugates are equal.
We enumerate Latin squares with conjugate symmetry and classify them according
to several common notions of equivalence. We also do similar enumerations
under additional hypotheses, such as assuming the Latin square is
reduced, diagonal, idempotent or unipotent.

Our data corrected an error in earlier literature and suggested
several patterns that we then found proofs for, including (1) The
number of isomorphism classes of semisymmetric idempotent Latin
squares of order $n$ equals the number of isomorphism classes of
semisymmetric unipotent Latin squares of order $n+1$, and (2) Suppose
$A$ and $B$ are totally symmetric Latin squares of order
$n\not\equiv0\bmod3$.  If $A$ and $B$ are paratopic then $A$ and $B$
are isomorphic.  
\end{abstract}

\section{Introduction}\label{Sintro}

A \emph{Latin square\/} is a matrix of order $n$ in which each row and
column is a permutation of some (fixed) symbol set of size $n$.
Throughout, we will assume that the symbol set is also used to index
the rows and columns. The symbols will be $\{1,2,\dots,n\}$ unless
specified otherwise. It is sometimes convenient to think of a Latin
square of order $n$ as a set of $n^2$ triples of the form (row,
column, symbol). The Latin property means that distinct triples never
agree in more than one coordinate. Latin squares are well known to be
equivalent to operation tables of finite \emph{quasigroups}. We will
usually state our results in terms of Latin squares but will
occasionally mention the corresponding interpretation in terms of
quasigroups.  See \cite{DKI,KD15} for background and terminology
regarding Latin squares and quasigroups.

For each Latin square there are six conjugate squares obtained by
uniformly permuting the coordinates of each triple.  These conjugates
can be labelled by a permutation giving the new order of the
coordinates, relative to the former order of $(1,2,3)$. For example,
the $(1,2,3)$-conjugate is the square itself and the
$(2,1,3)$-conjugate is its transpose.  We say that a square possesses
a \emph{conjugate symmetry} if at least two of the square's conjugates
are equal. The number of equal conjugates must be a divisor of 6.  A
square is said to be \emph{totally symmetric\/} if all six of its
conjugates are equal. It is \emph{semisymmetric} if it is equal to (at
least) three of its conjugates, which must necessarily include the
$(1,2,3)$, $(3,1,2)$ and $(2,3,1)$-conjugates. It is \emph{symmetric}
if it equals its $(2,1,3)$-conjugate. Any square which equals exactly
two of its conjugates has exactly two conjugates which are symmetric.

In this paper we count all of the Latin squares of small order
which have a conjugate symmetry. By the above comments, it is
sufficient to count the totally symmetric, semisymmetric and symmetric
Latin squares.
 
Let $\sym_n$ denote the symmetric group on $\{1,2,\dots,n\}$ and $\id$
denote the identity element in~$\sym_n$. The \emph{cycle structure} of
a permutation is a list of its cycle lengths in decreasing order,
using exponent notation to denote multiplicity. For example, one
permutation with the cycle structure $3^2\tdot2\tdot1^3$ is
$(1,2,3)(4,5,6)(7,8)(9)(10)(11)\in\sym_{11}$.  For any
$\alpha\in\sym_n$ we use $\ord(\alpha)$ to denote the \emph{order} of
$\alpha$ in $\sym_n$, which is the least common multiple of its cycle
lengths. We will write the image of $i$ under $\alpha$ as $i^\alpha$.

For $(\alpha,\beta,\gamma)\in\sym_n\times\sym_n\times\sym_n$ we can
apply $\alpha,\beta,\gamma$ to, respectively, the rows, columns and
symbols of a Latin square $L$. This operation is called
\emph{isotopism}. The resulting Latin square is \emph{isotopic} to $L$
and is written $L(\alpha,\beta,\gamma)$. If $\alpha=\beta=\gamma$ the
isotopism is an \emph{isomorphism}, if $\gamma=\id$ then the isotopism
is \emph{principal}, and if $\alpha=\beta$ then it is an
\emph{rrs-isotopism}.

Isotopism gives an action of $\sym_n\times\sym_n\times\sym_n$ on the
set of Latin squares of order $n$. On the same set there is also an
action of $\sym_n\wr\sym_3$, called \emph{paratopism}, which combines
an isotopism with taking a conjugate. The stabiliser of a Latin square
$L$ under isomorphism, isotopism and paratopism are, respectively, its
\emph{automorphism group}, \emph{autotopism group} and its
\emph{autoparatopism group}.  The set of squares isomorphic, isotopic,
rrs-isotopic and paratopic to $L$ are, respectively, the
\emph{isomorphism class}, \emph{isotopism class},
\emph{rrs-isotopism class} and \emph{species} of~$L$. Species are
sometimes known as \emph{main classes}. A useful observation that
follows directly from the definitions is:

\begin{lemma}\label{l:TSspeciso}
  The species of a totally symmetric Latin square contains a single
  isotopism class.
\end{lemma}

We gave two definitions of a Latin square above; first the standard
definition, then an equivalent definition in terms of triples. A third
way to think of an $n\times n$ Latin square is in terms of an $n\times
n\times n$ array of zeroes and ones. The $(i,j,k)$-th entry of the
array is zero unless the Latin square has symbol $k$ in column $j$ of
row $i$. The result is a 3-dimensional analogue of a permutation
matrix, which we call a \emph{permutation cube} (although the reader
is warned that this term has another meaning, see \cite{DKI}).

The conjugate symmetries studied in this paper each have a neat
geometric interpretation in terms of permutation cubes. Symmetric
Latin squares are those whose permutation cube possesses a reflective
symmetry in the plane through the vertices $(1,1,1)$, $(1,1,n)$,
$(n,n,1)$ and $(n,n,n)$ of the cube. Semisymmetric Latin squares are
those whose permutation cube possesses a threefold rotational symmetry
about the axis through vertices $(1,1,1)$ and $(n,n,n)$ of the cube.
Totally symmetric squares are those whose permutation cubes possess
both the above symmetries.

A \emph{reduced Latin square} is one in which the elements in the
first row and column occur in order. A reduced Latin square is the
operation table of a quasigroup in which 1 is an identity element. A
quasigroup with an identity element is called a \emph{loop}.  A
Latin square is {\em unipotent} if the symbols on the main diagonal
are all the same. It is \emph{diagonal} if the symbols on its main
diagonal are distinct.  It is \emph{idempotent} if it is diagonal and
the symbols occur in natural order down the main diagonal. More
generally, an element $i$ is idempotent if $i$ occurs in cell
$(i,i)$. At several subsequent points we will use that the number of
idempotent elements is an isomorphism invariant. We will also
repeatedly use the following simple consequences of the definitions.

\begin{lemma}\label{l:obvious}\mbox{ }
  \begin{enumerate}[label={\rm(\roman*)}]\itemsep=0pt
  \item\label{i:diagsym}
    A symmetric Latin square is diagonal if and only if it has odd order.
  \item\label{i:obvidemdiag} Idempotent Latin squares are necessarily diagonal.
  \item\label{i:obvuninotdiag} A unipotent Latin square of order
    $n\ge2$ cannot be diagonal or idempotent.
  \item\label{i:obvidemnotred} There are no reduced idempotent Latin
    squares of order $n\ge2$.
  \item\label{i:obvredsym} Every symmetric Latin square can be mapped
    to a reduced square by applying a unique permutation to the
    symbols.
  \item\label{i:obvredsemi} A semisymmetric reduced Latin square is
    necessarily unipotent.
  \end{enumerate}
\end{lemma}

\begin{proof}
Part \iref{i:diagsym} follows from the observation that off-diagonal
entries in a symmetric Latin square come in pairs. Parts
\iref{i:obvidemdiag} and \iref{i:obvuninotdiag} are immediate from the
definitions.  If $n\ge2$ then any reduced Latin square of order $n$
contains the triple $(1,2,2)$, which is incompatible with having the
triple $(2,2,2)$ that is required by idempotent Latin squares of order
$n$, hence \iref{i:obvidemnotred} holds.  Any symmetric Latin square
can be mapped to a reduced square by permuting the symbols to get the
first row in order (and hence also the first column in order).  This
operation preserves symmetry, which implies \iref{i:obvredsym}.

It remains to justify \iref{i:obvredsemi}.  Any reduced semisymmetric
Latin square of order $n$ has the triples $(1,i,i)$ for $1\le i\le n$
because it is reduced.  Semisymmetry then requires the triples
$(i,i,1)$ for $1\le i\le n$ to be present, which means the square is
unipotent.
\end{proof}

\lref{l:obvious}\iref{i:obvredsym} shows that the numbers of
rrs-isotopism classes, isotopism classes and species, respectively, of
symmetric Latin squares equal the numbers of rrs-isotopism classes,
isotopism classes and species containing symmetric reduced Latin
squares. The analogous property does not hold for isomorphism
classes, as we will see in \Tref{T:comloop}.

\Tref{T:summary} presents a summary of the classes that are enumerated
in this paper. In order to read the table, the following notes are important.
\begin{itemize}[itemsep=0pt]
\item There are 3 sections of the table, one for each of the conjugate
  symmetries (symmetric, semisymmetric and totally symmetric).
\item The first row in each section covers the conjugate symmetry
  without any further assumptions. Subsequent rows in the section add
  extra assumptions, as specified in the first column. We refer to the
  conjugate symmetry together with any further restrictions added in
  the first column, as the ``category'' of Latin squares being
  counted.
\item Columns (after the first) specify what classes are being counted,
  whether it is isomorphism classes, rrs-isotopism classes, isotopism
  classes, species or all squares. Another viewpoint for the last
  column is that it counts ``labelled'' Latin squares, meaning that
  they count as different unless they are equal as matrices.
\item Each entry in the table gives a reference which says which table
  to look at in this paper for an enumeration of the class (as
  specified at the top of the column) of Latin squares of order $n$ in
  the category (as determined by the row).  In many cases the
  reference includes mention of a result or results which justify our
  claim that the indicated table contains the desired numbers.
\item
  Where a reference is given to Table X, the column of Table X that
  should be consulted will have the same heading as the column
  containing the reference, unless the reference specifies another
  column name.
\item
  In some cases the reference includes a \red\ which means that the
  order needs to be shifted down by 1.  In other words, to get the
  numbers for order $n$ Latin squares (in the desired category), you should
  look at order $n-1$ in the referenced table.
\item
  Superscripts on the category title indicate further information as
  follows:
  \begin{itemize}
    \item[] \odd \ means that category is nonempty only for odd orders.    
    \item[] \even \ means that category is nonempty only for even orders.
    \item[] \unip \ means that squares in that category are
      necessarily unipotent, by \lref{l:obvious}\iref{i:obvredsemi}.
  \end{itemize}
  The first two of these are an important caveat when following a
  reference for a category with that superscript, and the third is
  helpful when understanding the justification we provide for such a
  reference.
\end{itemize}

Here and throughout our paper we adopt a convention for the many
tables that present counts of classes of Latin squares with particular
properties. A class will be counted if it contains any Latin square
with the named properties. So, for example, a column headed
``species'' in a table of symmetric Latin squares would count any
species containing at least one symmetric Latin square, even though
only some of the Latin squares in the species are symmetric. In text
we shorten ``class containing at least one'' to ``class containing''
and we use ``class of'' for the case when every object in the class
has the pertinent property.

\Tref{T:summary} covers all combinations of interest among these
properties: symmetric, semi\-symmetric, totally symmetric, reduced,
diagonal, idempotent, and unipotent.  Guided by \lref{l:obvious}, we
do not list the following categories in \Tref{T:summary}: idempotent
diagonal, idempotent reduced, idempotent unipotent or unipotent
diagonal. Moreover, for semisymmetric and totally symmetric Latin
squares we do not combine reduced with any of unipotent, idempotent or
diagonal. Nor did we count rrs-isotopism classes except for symmetric
Latin squares.  The notion of rrs-isotopism is particular to the
symmetric case since it is the strongest form of paratopism which
respects symmetry.  It could be argued that, by the same logic, we
should not count isotopism classes or species for any of our conjugate
symmetries. However, isotopism classes and species are extremely
widely studied notions in the Latin squares literature. Also, in the
authors' experience, isotopism classes or species containing Latin
squares with conjugate symmetry are often extremal for properties
which have no apparent connection to their symmetry.  See
\cite{EW12,cycsw} for a number of examples. This makes catalogues of
isotopism classes or species that contain Latin squares with conjugate
symmetry useful.

\begin{landscape}
\begin{table}[p]
\small{%
\centering
\def\qd{\kern 0.0em }
\newdimen\digitwidth \setbox0=\hbox{\rm0} \digitwidth=\wd0
\catcode`@=\active \def@{\kern\digitwidth} \centerline{
\vbox{\offinterlineskip \hrule 
\halign{&\vrule#&\strut
        \qd\hfil\mybox{#}\hfil\qd&\vrule#&
        \qd\hfil\mybox{#}\hfil\qd&\vrule#&
        \qd\hfil\mybox{#}\hfil\qd&\vrule#&
        \qd\hfil\mybox{#}\hfil\qd&\vrule#&
        \qd\hfil\mybox{#}\hfil\qd&\vrule#&
        \qd\hfil\mybox{#}\hfil\qd&\vrule#\cr
height2pt&\omit&&&&&&&&&&&\cr
&&&isomorphism\\ classes&&rrs-isotopism\\ classes&&isotopism\\ classes&&species&&all squares&\cr
\hline
&{\bf Symmetric:}\hfill\hfill&&\Tref{T:comloop}&&\Tref{tabsym}&&\Tref{tabsym} species\\ \lref{l:isotsymm}&&\Tref{tabsym}&&$n!$(\Tref{tabsym} reduced)\\\lref{l:obvious}\iref{i:obvredsym}&\cr
\cline{3-12}
&idempotent\odd&&\Tref{tabsym} species\\ \tref{t:idemSLS}\iref{i:idemiso}&&\Tref{tabsym} species\\ \tref{t:idemSLS}\iref{i:idemrrs}&&\Tref{tabsym} species\\ \tref{t:idemSLS}\iref{i:idemisot}&&\Tref{tabsym}\\ \tref{t:idemSLS}\iref{i:idemsp}&&\Tref{tabsym} reduced\\ \tref{t:uniSLS}\iref{u:prolong}&\cr
\cline{3-12}
&unipotent\even&&\Tref{T:comloop}\red\\ \tref{t:uniSLS}\iref{u:prolongiso}&&\Tref{T:unipsym}\\ \tref{t:uniSLS}\iref{u:allhavered}&&\Tref{T:unipsym} species\\ \tref{t:uniSLS}\iref{u:specisot},\iref{u:allhavered}&&\Tref{T:unipsym}\\ \tref{t:uniSLS}\iref{u:allhavered}&&$n!$(\Tref{tabsym} reduced\red)\\ \tref{t:uniSLS}\iref{u:prolonger}&\cr
\cline{3-12}
&reduced\\ unipotent\even&&\Tref{T:unipsym}&&\Tref{T:unipsym}&&\Tref{T:unipsym} species\\ \lref{l:isotsymm}&&\Tref{T:unipsym}&&\Tref{tabsym} reduced\red\\ \tref{t:uniSLS}\iref{u:prolong}&\cr
\cline{3-12}
&reduced&&\Tref{T:comloop} loops&&\Tref{tabsym}\\\lref{l:obvious}\iref{i:obvredsym}&&\Tref{tabsym} species\\Lemmas \ref{l:obvious}\iref{i:obvredsym}, \ref{l:isotsymm}&&\Tref{tabsym}\\\lref{l:obvious}\iref{i:obvredsym}&&\Tref{tabsym} reduced&\cr
\cline{3-12}
&diagonal\odd&&\Tref{T:comloop}\\\lref{l:obvious}\iref{i:diagsym}&&\Tref{tabsym}\\\lref{l:obvious}\iref{i:diagsym}&&\Tref{tabsym} species\\ Lemmas \ref{l:obvious}\iref{i:diagsym}, \ref{l:isotsymm}&&\Tref{tabsym}\\\lref{l:obvious}\iref{i:diagsym}&&$n!$(\Tref{tabsym} reduced)\\\lref{l:obvious}\iref{i:diagsym},\iref{i:obvredsym}&\cr
\cline{3-12}
&reduced\\ diagonal\odd&&\Tref{T:comloop} loops\\\lref{l:obvious}\iref{i:diagsym}&&\Tref{tabsym}\\\lref{l:obvious}\iref{i:diagsym},\iref{i:obvredsym}&&\Tref{tabsym} species\\Lemmas \iref{l:obvious}\iref{i:diagsym},\iref{i:obvredsym}, \ref{l:isotsymm}&&\Tref{tabsym}\\\lref{l:obvious}\iref{i:diagsym},\iref{i:obvredsym}&&\Tref{tabsym} reduced\\\lref{l:obvious}\iref{i:diagsym}&\cr
\hline\hline
&{\bf Semi-}\\{\bf symmetric:}\hfill\hfill&&\Tref{tabsemi}&&-&&\Tref{tabsemi}&&\Tref{tabsemi}&&\Tref{tabsemi}&\cr
\cline{3-12}
&idempotent&&\Tref{T:idem}&&-&&\Tref{T:idem}&&\Tref{T:idem}&&\Tref{T:idem}&\cr
\cline{3-12}
&unipotent&&\Tref{T:semisymloop}\\\lref{l:semiuniloop}&&-&&\Tref{T:semisymloop}\\\lref{l:semiuniloop}&&\Tref{T:semisymloop}\\\lref{l:semiuniloop}&&$n$(\Tref{T:semisymloop})\\\lref{l:semiuniloop}&\cr
\cline{3-12}
&reduced\unip&&\Tref{T:semisymloop}&&-&&\Tref{T:semisymloop}&&\Tref{T:semisymloop}&&\Tref{T:semisymloop}&\cr
\cline{3-12}
&diagonal&&\Tref{T:diag}&&-&&\Tref{T:diag}&&\Tref{T:diag}&&\Tref{T:diag}&\cr
\hline\hline
&{\bf Totally}\\{\bf Symmetric:}\hfill&&\Tref{tabtotsym}&&-&&\Tref{tabtotsym} species\\\lref{l:TSspeciso}&&\Tref{tabtotsym}&&\Tref{tabtotsym}&\cr
\cline{3-12}
&idempotent\odd&&\Tref{tabSTS} species\\\tref{t:STS}\iref{i:TSidemisom}&&-&&\Tref{tabSTS} species\\\tref{t:STS}\iref{i:TSidemisot}&&\Tref{tabSTS}&&\Tref{tabSTS}&\cr
\cline{3-12}
&unipotent\even&&\Tref{tabSTS} species\red\\\tref{t:STS}\iref{i:TSuniisom}&&-&&\Tref{tabSTS} species\red\\\tref{t:STS}\iref{i:TSuniisot}&&\Tref{tabSTS}\red\\\tref{t:STS}\iref{i:TSunispec}&&$n$(\Tref{tabSTS}\red)\\\lref{l:semiuniloop}&\cr
\cline{3-12}
&reduced\unip\even&&\Tref{tabSTS} species\red\\\tref{t:STS}\iref{i:TSredisom}&&-&&\Tref{tabSTS} species\red\\\tref{t:STS}\iref{i:TSredisot}&&\Tref{tabSTS}\red\\\tref{t:STS}\iref{i:TSredspec}&&\Tref{tabSTS}\red\\\tref{t:uniSLS}\iref{u:prolongTS}&\cr
\cline{3-12}
&diagonal\odd&&\Tref{tabtotsym}\\\lref{l:obvious}\iref{i:diagsym}&&-&&\Tref{tabtotsym} species\\Lemmas \ref{l:TSspeciso}, \ref{l:obvious}\iref{i:diagsym}&&\Tref{tabtotsym}\\\lref{l:obvious}\iref{i:diagsym}&&\Tref{tabtotsym}\\\lref{l:obvious}\iref{i:diagsym}&\cr
\hline
}}}}
\caption{\label{T:summary}Summary of the results of this paper}
\end{table}
\end{landscape}

The structure of this paper is as follows.  We report the results of
our enumerations of symmetric, semisymmetric and totally symmetric
Latin squares in \sref{s:sym}, \sref{s:semisym}, and \sref{s:totsym},
respectively. In each case, we will also count the Latin squares with
conjugate symmetries that have the additional properties of being
unipotent, idempotent or diagonal.
In \sref{s:Sade} we explain how our results from \sref{s:semisym}
uncovered an error in earlier literature.

All numbers reported in this paper were computed independently by the
two authors using algorithms that differed in some details. The total
CPU time taken for all of our computations ran to several months.  For
each problem we also computed small order catalogues by elementary
direct searches, to crosscheck the more sophisticated algorithms which
we needed for larger cases. We have made catalogues of many of the
Latin squares that we generated in our enumerations available online
\cite{WWWW}.

\section{Symmetric Latin squares}\label{s:sym}

In this section we count and classify symmetric Latin squares.
\Tref{tabsym} shows data for the symmetric Latin squares of order up
to 13, classified by rrs-isotopism and species. The last column counts
all reduced symmetric Latin squares.  By
\lref{l:obvious}\iref{i:obvredsym}, the number of all symmetric Latin
squares of order $n$ can be obtained by multiplying the number of
reduced symmetric Latin squares of order $n$ by $n!$, the number of
ways to permute the symbols.

\begin{table}
\centering
\def\qd{\kern 0.23em }
\newdimen\digitwidth \setbox0=\hbox{\rm0} \digitwidth=\wd0
\catcode`@=\active \def@{\kern\digitwidth} \centerline{
\vbox{\offinterlineskip \hrule 
\halign{&\vrule#&\strut
        \qd\hfil#\hfil\qd&\vrule#&
        \qd\hfil#\hfil\qd&\vrule#&
        \qd\hfil#\hfil\qd&\vrule#&
        \qd\hfil#\hfil\qd&\vrule#&
        \qd\hfil#\hfil\qd&\vrule#\cr
height2pt&\omit&&&&&&&\cr
&\mybox{order}&&\mybox{rrs-isotopism\\classes}&&\mybox{species}&&\mybox{reduced
}&\cr
height2pt&\omit&&&&&&&\cr
\noalign{\hrule height0.8pt}
height2pt&\omit&&&&&&&\cr
&2&&1&&1&&1&\cr
&3&&1&&1&&1&\cr
&4&&2&&2&&4&\cr
&5&&1&&1&&6&\cr
&6&&6&&6&&456&\cr
&7&&7&&7&&6240&\cr
&8&&423&&415&&10936320&\cr
&9&&3460&&3460&&1225566720&\cr
&10&&35878510&&35878418&&130025295912960&\cr
&11&&6320290037&&6320290037&&252282619805368320&\cr
&12&&4612966007179768&&4612965997149292&&2209617218725251404267520&\cr
&13&&15859695832489637513&&15859695832489637513&&98758655816833727741338583040&\cr
height2pt&\omit&&&&&&&\cr
}\hrule}}
\caption{\label{tabsym}Counts of symmetric Latin squares. 
The number of isotopism classes equals the number of species by 
\lref{l:isotsymm}. Also, for odd orders, the number of rrs-isotopism
classes equals the number of species, by the same result.}
\end{table}

We now explain how the numbers in \Tref{tabsym} were obtained,
starting with the numbers in the final column. These were computed
using a simple adaptation of the method used in \cite{MW05} to count
the Latin squares of order 11. Rather than building symmetric Latin
squares we built the (equinumerous) Latin squares that equal their
$(1,3,2)$-conjugate. This was achieved by adding one row at a time,
ensuring the new row was an involution (when viewed as a permutation
in image format). An alternative viewpoint is that we counted
factorisations by adding a factor at a time. In \cite{MW05} the task
was to factorise the complete bipartite graph $K_{n,n}$ into
1-factors.  In the present work we factorised the complete graph $K_n$
with a loop added to each vertex into factors consisting of 
disjoint edges which may be loops. In the case of odd $n$ each factor
had to have exactly one loop (cf.~\lref{l:obvious}\iref{i:diagsym}).

For any odd order $n$ the number of symmetric reduced Latin squares of
order $n$ is equal to the number of 1-factorisations of $K_{n+1}$ (see
e.g.~\cite{WI05}). These numbers have been computed up to $n=13$.  We
did not recompute the number for $n=13$, but instead relied on the
result quoted in \cite{KO09}. For smaller odd $n$ we did recompute the
numbers, and used the previously published values as a validation of
our code.

Next we counted the rrs-isotopism classes of symmetric Latin squares.
These numbers can be inferred by generating all symmetric Latin
squares that have a non-trivial rrs-autotopism. Any such Latin square
$L$ necessarily possesses an rrs-autotopism $(\alpha,\alpha,\gamma)$
of prime order.  The same autotopism combines with
$(2,1,3)$-conjugation to produce an autoparatopism, since $L$ is
symmetric.  We used the classifications of autotopisms
\cite{SVW12} and autoparatopisms \cite{MW17}. From those lists we
deduce that $\alpha$ and $\gamma$ must have one of the cycle structures
given in \Tref{T:agopts} and that we are at liberty to fix any choice
of $\alpha$ and $\gamma$ with the appropriate cycle
structure.  For each possible $(\alpha,\gamma)$, we generated the
number of symmetric Latin squares having $(\alpha,\alpha,\gamma)$ as
an autotopism.  These numbers are shown in \Tref{T:agopts} under the
heading ``LS''.

\def\hhline{\hline\multicolumn{3}{c}{}\\[-1.6ex]\hline}

\begin{table}
  \[\begin{array}{cccc}
  \begin{array}{|l|l|r|}
    \hline
    \multicolumn{3}{|c|}{\tspacer n=2}\\
    \hline
\alpha&\gamma&\text{LS}\\
\hline
2&1^2&1\\

\hhline

\multicolumn{3}{|c|}{\tspacer n=3}\\
    \hline
\alpha&\gamma&\text{LS}\\
\hline
2\tdot1&2\tdot1&1\\
3&3&1\\

\hhline

\multicolumn{3}{|c|}{\tspacer n=4}\\
    \hline
\alpha&\gamma&\text{LS}\\
\hline
\tspacer
2^2&2\tdot1^2&1\\
&1^4&2\\
2\tdot1^2&2\tdot1^2&2\\
3\tdot1&3\tdot1&1\\

\hhline

\multicolumn{3}{|c|}{\tspacer n=5}\\
    \hline
\alpha&\gamma&\text{LS}\\
\hline
\tspacer
2^2\tdot1&2^2\tdot1&1\\
5&5&1\\

\hhline

\multicolumn{3}{|c|}{\tspacer n=6}\\
    \hline
\alpha&\gamma&\text{LS}\\
\hline
\tspacer
2^3&2^2\tdot1^2&2\\
&2\tdot1^4&4\\
&1^6&2\\
2^2\tdot1^2&2^2\tdot1^2&6\\
3^2&3^2&2\\
&3\tdot1^3&2\\
5\tdot1&5\tdot1&1\\
\hline
  \end{array}
  \quad
  \begin{array}{|l|l|r|}
    \hline
\multicolumn{3}{|c|}{\tspacer n=7}\\
    \hline
\alpha&\gamma&\text{LS}\\
\hline
\tspacer
2^3\tdot1&2^3\tdot1&2\\
2^2\tdot1^3&2^2\tdot1^3&4\\
3^2\tdot1&3^2\tdot1&5\\
7&7&2\\

\hhline

\multicolumn{3}{|c|}{\tspacer n=8}\\
    \hline
\alpha&\gamma&\text{LS}\\
\hline
\tspacer
2^4&2^3\tdot1^2&33\\
&2^2\tdot1^4&131\\
&2\tdot1^6&96\\
&1^8&44\\
2^3\tdot1^2&2^3\tdot1^2&26\\
2^2\tdot1^4&2^2\tdot1^4&46\\
3^2\tdot1^2&3^2\tdot1^2&23\\
7\tdot1&7\tdot1&2\\

\hhline

 \multicolumn{3}{|c|}{\tspacer n=9}\\
    \hline
\alpha&\gamma&\text{LS}\\
\hline
\tspacer
2^4\tdot1&2^4\tdot1&39\\
2^3\tdot1^3&2^3\tdot1^3&101\\
3^3&3^3&13\\
3^2\tdot1^3&3^2\tdot1^3&20\\
\hline
  \end{array}
  \quad
\begin{array}{|l|l|r|}
  \hline
  \multicolumn{3}{|c|}{\tspacer n=10}\\
    \hline
\alpha&\gamma&\text{LS}\\
\hline
  \tspacer
2^5&2^4\tdot1^2&1784\\
&2^3\tdot1^4&32144\\
&2^2\tdot1^6&37784\\
&2\tdot1^8&7488\\
&1^{10}&252\\
2^4\tdot1^2&2^4\tdot1^2&9525\\
2^3\tdot1^4&2^3\tdot1^4&5434\\
3^3\tdot1&3^3\tdot1&242\\
3^2\tdot1^4&3^2\tdot1^4&67\\
5^2&5^2&15\\
5^2&5\tdot1^5&2\\

\hhline

\multicolumn{3}{|c|}{\tspacer n=11}\\
    \hline
\alpha&\gamma&\text{LS}\\
\hline
  \tspacer
2^5\tdot1&2^5\tdot1&11352\\
2^4\tdot1^3&2^4\tdot1^3&144592\\
2^3\tdot1^5&2^3\tdot1^5&0\\
5^2\tdot1&5^2\tdot1&210\\
11&11&5\\
\hline
\end{array}
\quad
\begin{array}{|l|l|r|}
  \hline
  \multicolumn{3}{|c|}{\tspacer n=12}\\
    \hline
\alpha&\gamma&\text{LS}\\
\hline
  \tspacer
2^6&2^5\tdot1^2&14530952\\
&2^4\tdot1^4&470822508\\
&2^3\tdot1^6&1556098547\\
&2^2\tdot1^8&1216169007\\
&2\tdot1^{10}&294114559\\
&1^{12}&20147679\\
2^5\tdot1^2&2^5\tdot1^2&36584824\\
2^4\tdot1^4&2^4\tdot1^4&64350427\\
2^3\tdot1^6&2^3\tdot1^6&1463416\\
3^4&3^4&44149\\
&3^3\tdot1^3&127621\\
&3^2\tdot1^6&11400\\
&3\tdot1^9&168\\
3^3\tdot1^3&3^3\tdot1^3&9867\\
5^2\tdot1^2&5^2\tdot1^2&2394\\
11\tdot1&11\tdot1&5\\

\hhline

\multicolumn{3}{|c|}{\tspacer n=13}\\
    \hline
\alpha&\gamma&\text{LS}\\
\hline
\tspacer
2^6\tdot1&2^6\tdot1&183778440\\
2^5\tdot1^3&2^5\tdot1^3&4076414984\\
2^4\tdot1^5&2^4\tdot1^5&144762344\\
3^4&3^4&9005726\\
5^2\tdot1^3&5^2\tdot1^3&11364\\
13&13&14\\
\hline
\end{array}
\end{array}
\]
\caption{\label{T:agopts}
  The number of rrs-isotopism classes containing symmetric Latin
  squares with autotopism $(\alpha,\alpha,\gamma)$ where $\alpha$ and
  $\gamma$ are permutations with the given cycle structure.
}
\end{table}

It is interesting that only one of the counts in \Tref{T:agopts} is
zero.  The classifications in \cite{SVW12} and \cite{MW17} give
necessary and sufficient conditions for autotopisms and
autoparatopisms to be achievable.  However, when we insist on
simultaneously achieving an autotopism and a separate autoparatopism,
the conditions are necessary but no longer sufficient. It is not hard
to see that no symmetric Latin square can achieve
$(\alpha,\alpha,\gamma)$ as an autotopism when
$\alpha$ and $\gamma$ both have cycle structure $2^3\tdot1^5$.
After filling in three of the symbols that are fixed by $\gamma$ there
is nowhere to place the other two such symbols.

For each entry in \Tref{T:agopts} we used one choice of permutations
$\alpha$ and $\gamma$ with the indicated cycle structures, and generated all
symmetric Latin squares possessing $(\alpha,\alpha,\gamma)$ as an
autotopism. Each time we built a square we calculated the order of its
rrs-autotopism group. If that group had order 2 and $n\ge12$ then we
counted the square but did not store it. All other generated squares
were stored. Throwing away the squares with rrs-autotopism group size
2 saved substantial disk space. For example, when $(\alpha,\gamma)$ had
cycle structure $(2^6,2^3\cdot1^6)$ there were 1553860785 squares with
group size 2, but only 2237762 squares with a larger group.
Crucially, the only way two squares from our enumeration could be
rrs-isotopic to each other is if they had a rrs-isotopism group of
order greater than 2 (since, by assumption, they possess two different
non-trivial rrs-autotopisms). Thus with the aid of our stored
catalogue we were able to find representatives of all rrs-isotopism
classes for which the rrs-autotopism group had order greater than
2. Combining with the count of all the squares we generated, we could
then infer the number of rrs-isotopism classes with an autoparatopism
group of order greater than 2, and the number of Latin squares in
those classes.  The number of rrs-isotopism classes for which the
autoparatopism group had order exactly 2 could then be inferred from
the last column of \Tref{tabsym}.

Our next task was to count species. To do that,
we will need the following result, which is an immediate consequence
of \cite[Thm~1]{MMM07}.

\begin{lemma}\label{l:class}
  If $(\alpha,\beta,\id)$ is a principal autotopism of some Latin
  square and $\alpha\ne\id$, then $\alpha$ and $\beta$ have the same
  cycle structure, and neither has fixed points.
\end{lemma}

Define $\Omega_n$ to be the set of symmetric Latin squares of order
$n$ that possess an autotopism of the form $(\theta,\theta^{-1},\id)$
where $\theta\in\sym_n$ is semiregular of prime order (meaning
$\theta$ has no fixed points and each cycle of $\theta$ has the same
length $p$, where $p$ is prime).

\begin{lemma}\label{l:isotsymm}
Suppose $A$ and $B$ are paratopic symmetric Latin squares of order $n$.  
Then $A$ and $B$ are isotopic. Also, if $A$ and $B$ are 
not rrs-isotopic then $n$ is even and $A,B\in\Omega_n$. 
\end{lemma}

\begin{proof}
The fact that $A$ and $B$ are isotopic follows immediately from
\cite[Lem.\,15]{WI05}. So assume that $A(\alpha,\beta,\gamma)=B$. Since
both $A$ and $B$ are symmetric, we also have
$A(\beta,\alpha,\gamma)=B$. Hence
\begin{equation}\label{e:nontrivatp}
B=A(\alpha,\beta,\gamma)
=B(\beta^{-1},\alpha^{-1},\gamma^{-1})(\alpha,\beta,\gamma)
=B(\beta^{-1}\alpha,\alpha^{-1}\beta,\id).
\end{equation}
Now either $\alpha=\beta$, in which case $A$ and $B$ are rrs-isotopic,
or $B$ has a non-trivial autotopism of the form
$(\theta,\theta^{-1},\id)$ by \eref{e:nontrivatp}.  By replacing
$\theta$ by an appropriate power of $\theta$, we may assume that it
has prime order.  It then follows from \lref{l:class} that it is
semiregular. Hence $B\in\Omega_n$.  The proof that $A\in\Omega_n$ is
similar.

Finally, suppose that $n$ is odd. In that case, by \lref{l:obvious}\iref{i:diagsym},
we can convert $A,B$ to idempotent symmetric Latin squares
$A',B'$ by permuting symbols. By \cite[Lem.\,6]{WI05}, $A'$ is
isomorphic to $B'$. It follows that $A$ is rrs-isotopic to $B$.
\end{proof}

\lref{l:isotsymm} allowed us to count the species that contain
symmetric Latin squares of order $n$ as follows. We may assume that
$n$ is even, since otherwise
the number of species equals the number of rrs-classes, which we
have already counted. We generated $\Omega_n$ by considering the
primes $p$ that divide $n$. For each such prime we generated the
symmetric Latin squares that have an autotopism
$(\theta,\theta^{-1},\id)$ where $\theta$ is one fixed semiregular permutation
of order $p$ 
(in the case $p=2$ this is a task we have already done when
counting rrs-isotopism classes, given that $\theta^{-1}=\theta$).
The hardest case we had to handle was
$n=12$. In that case, we only need to consider $p=2$ and $p=3$.  By
\lref{l:isotsymm}, among the symmetric Latin squares of order $n$ the
number of species overall is the number of species in $\Omega_n$ plus
the number of rrs-isotopism classes outside $\Omega_n$. We got the
latter number indirectly, by counting the number of rrs-isotopism
classes in $\Omega_n$ and subtracting them from the total, which we
had already calculated. In this way, we counted the species just by
examining $\Omega_n$. Some care was required in the hardest case, when
$n=12$, because we had not kept the symmetric Latin squares with an
autotopism with cycle structure $(2^6,2^6,\id)$ unless they had an additional
rrs-autotopism. These discarded squares had the potential to appear in our
catalogue of the squares with a principal autotopism
$(\theta,\theta^{-1},\id)$ of order 3. Indeed, in that catalogue there were
40 rrs-isotopism classes (34 species) of symmetric Latin squares
that also possessed a principal autotopism of order 2. All but one
of those classes had been discarded due to having no additional rrs-autotopism.

To be extra careful, we also computed the number of species a second
(slower) way. In this approach we generated every symmetric Latin
square with one of the autotopisms in \Tref{T:agopts}, but this time
screened them for isotopism. We counted the isotopism classes with
autotopism group of order exactly 2, but discarded their
representatives. We stored representatives of all classes where the
order of the autotopism group exceeded 2 for further comparison.  In
this and other respects, the computation was similar to how we counted
rrs-isotopism classes. The result gave us independent confirmation of
the number of isotopism classes (which equals the number of species).

It is clear from \Tref{tabsym} that $n=8$ is the smallest
order for which there are isotopic symmetric Latin squares that are
not rrs-isotopic. An example of this behaviour is the following pair of 
Latin squares:
\[
\left[\begin{array}{cccccccc}
1&2&3&4&5&6&7&8\\
2&1&4&3&6&5&8&7\\
3&4&1&2&7&8&5&6\\
4&3&2&1&8&7&6&5\\
5&6&7&8&2&1&3&4\\
6&5&8&7&1&3&4&2\\
7&8&5&6&3&4&2&1\\
8&7&6&5&4&2&1&3\\
\end{array}\right]\qquad\qquad
\left[\begin{array}{cccccccc}
1&2&3&4&5&6&7&8\\
2&1&4&3&6&5&8&7\\
3&4&1&2&7&8&5&6\\
4&3&2&1&8&7&6&5\\
5&6&7&8&1&2&4&3\\
6&5&8&7&2&3&1&4\\
7&8&5&6&4&1&3&2\\
8&7&6&5&3&4&2&1\\
\end{array}\right]
\]
Applying the isotopism $\big((1423)(78),(12)(5768),(1423)(56)\big)$
to the left hand square, we see that the two squares are isotopic.
However, they are not rrs-isotopic, since the right hand square has
a symbol that appears six times on the main diagonal and the left hand
square has no such symbol.

Next, we turn to the issue of counting isomorphism classes of
symmetric Latin squares (in other words, counting commutative quasigroups
up to isomorphism). For this task we made use of the following result, which
allows us to count isomorphism classes using the catalogue of rrs-isotopism
classes that we had computed above. For permutations $\alpha,\beta\in\sym_n$
we write $\alpha^\beta$ as shorthand for $\beta^{-1}\alpha\beta$.

\begin{theorem}\label{t:rrs-isom}
 Let $L$ be a symmetric Latin square, and $\varGamma$ its rrs-autotopism group.
 Then the number of isomorphism classes of symmetric Latin squares
 rrs-isotopic to $L$ is 
 \[
 \frac{1}{\card\varGamma} 
 \sum_{\smash{(\alpha,\alpha,\beta)\in}\varGamma} \psi(\alpha,\beta),
 \]
 where
 \[
 \psi(\alpha,\beta) =
 \begin{cases}
   \prod_{i=1}^k (n_i)!\,c_i^{n_i}, & \text{if $\alpha,\beta$ have the
     same cycle structure $c_1^{n_1}\cdots c_k^{n_k}$;} \\
   0, & \text{otherwise}.
 \end{cases}
 \]
\end{theorem}

\begin{proof}
  The rrs-isotopism class of $L$ is
   $\mathcal{L} = \{ L(\sigma,\sigma,\tau) :  \sigma,\tau\in \sym_n\}$.
  Define $\varLambda=\{ (\rho,\rho,\rho) : \rho\in \sym_n\}$.
  The number $N$ of isomorphism classes in $\mathcal{L}$ is the number
  of orbits of the action of $\varLambda$ on $\mathcal{L}$.
  By the Frobenius--Burnside Lemma, we have
  \begin{align*}
      N &= \frac{1}{n!}\, 
       \sum_{\rho\in \sym_n}\,\Card{ \{M\in\mathcal{L} : M(\rho,\rho,\rho) = M\} } \\
   &= \frac{1}{n!\,\card\varGamma}\,
      \Card{ \{ \sigma,\tau,\rho\in \sym_n 
                     : L(\sigma,\sigma,\tau)(\rho,\rho,\rho) =L(\sigma,\sigma,\tau) \} }\\
   &= \frac{1}{n!\,\card\varGamma}\,
       \Card{ \{ \sigma,\tau,\rho\in \sym_n 
       : (\rho^{\sigma^{-1}},\rho^{\sigma^{-1}},\rho^{\tau^{-1}})\in\varGamma\} }
          \displaybreak[1]\\
   &= \frac{1}{n!\,\card\varGamma} \sum_{(\alpha,\alpha,\beta)\in\varGamma}
        \Card{ \{ \sigma,\tau\in \sym_n : \alpha^{\sigma}=\beta^{\tau} \} } \\
   &= \frac{1}{\card\varGamma} \sum_{(\alpha,\alpha,\beta)\in\varGamma} 
         \Card{ \{ \sigma\in \sym_n : \alpha^\sigma=\beta \} }
       = \frac{1}{\card\varGamma} 
         \sum_{(\alpha,\alpha,\beta)\in\varGamma} \psi(\alpha,\beta).\qedhere
  \end{align*}
\end{proof}

\begin{table}[htb]
\centering
\def\qd{\kern 0.53em }
\newdimen\digitwidth \setbox0=\hbox{\rm0} \digitwidth=\wd0
\catcode`@=\active \def@{\kern\digitwidth} \centerline{
\vbox{\offinterlineskip \hrule 
\halign{&\vrule#&\strut
        \qd\hfil#\hfil\qd&\vrule#&
        \qd\hfil#\hfil\qd&\vrule#&
        \qd\hfil#\hfil\qd&\vrule#\cr
height2pt&\omit&&&&&\cr
&\mybox{order}&&\mybox{isomorphism\\classes}&&\mybox{loops}&\cr
height2pt&\omit&&&&&\cr
\noalign{\hrule height0.8pt}
height2pt&\omit&&&&&\cr
&2&&1&&1&\cr
&3&&3&&1&\cr
&4&&7&&2&\cr
&5&&11&&1&\cr
&6&&491&&8&\cr
&7&&6381&&17&\cr
&8&&10940111&&2265&\cr
&9&&1225586965&&30583&\cr
&10&&130025302505741&&358335026&\cr
&11&&252282619993126717&&69522550106&\cr
&12&&2209617218725712597768722&&55355570223093935&\cr
&13&&98758655816833782283724345637&&206176045800229002160&\cr
height2pt&\omit&&&&&\cr
}\hrule}}
\caption{\label{T:comloop}Counts of symmetric Latin squares and commutative loops up to isomorphism.}
\end{table}

\tref{t:rrs-isom} allowed us to complete the middle column of
\Tref{T:comloop}. In the last column of the same table we count
isomorphism classes of symmetric reduced Latin squares (which are
equinumerous with the isomorphism classes of commutative loops).  To
do that, we applied the following result to the representatives of
rrs-isotopism classes of symmetric Latin squares, which
we found in the production of \Tref{tabsym}.

\begin{theorem}\label{t:loops}
 Let $L$ be a symmetric Latin square, and $\varGamma$ its rrs-autotopism group.
 Then the number of isomorphism classes containing reduced symmetric Latin squares
 rrs-isotopic to $L$ is
 \[
 \frac{1}{\card\varGamma} 
 \sum_{\smash{(\alpha,\alpha,\beta)\in}\varGamma} \lambda(\alpha,\beta),
 \]
 where $\lambda(\alpha,\beta)$ is the number of fixed points of $\alpha$,
 if $\alpha$ and $\beta$ have the same cycle structure, and
 $\lambda(\alpha,\beta)=0$ otherwise.
\end{theorem}

\begin{proof}
  First note that any isomorphism between reduced Latin squares must fix
  $(1,1,1)$. Moreover, any isomorphism that fixes $(1,1,1)$ preserves
  the property of being reduced.
Let $H=\{(\rho,\rho,\rho) : \rho\in\sym_n \text{~and~} 1^\rho=1\}$ and consider $H$ acting on the set of symmetric reduced Latin squares rrs-isotopic to~$L$.
Each orbit of this action is the set of reduced Latin squares within
one of the isomorphism classes that we wish to count.
By the Frobenius--Burnside Lemma, $|H|$ times the number of orbits is equal to
the number of distinct triples $(M,\gamma,\gamma)$ such that $M$ is a reduced square
rrs-isotopic to $L$ and $(\gamma,\gamma,\gamma)$ is an automorphism of~$M$.
These triples have the form
$(L(\sigma,\sigma,\tau),\alpha^\sigma,\beta^\tau)$ where
$(\alpha,\alpha,\beta)\in\varGamma$, $\alpha^\sigma=\beta^\tau$
and $\sigma,\tau\in\sym_n$.

Now consider fixed $(\alpha,\alpha,\beta)\in\varGamma$ and consider
the ways of choosing $\sigma,\tau\in\sym_n$.
Since we need $\alpha^\sigma=\beta^\tau$, it must be that $\alpha$ and
$\beta$ have the same cycle structure, so we assume this is the case.
Suppose $L=(\ell_{ij})$ and
define $\xi_i\in\sym_n$ by $j^{\xi_i}=\ell_{ij}$ for each~$j$.
Let $\delta\in \sym_n$ fix~1.
Then for $\sigma=\xi_i(1\,\ell_{ii})\delta$ and $\tau= (1\,\ell_{ii})\delta$,
the rrs-isotopism $(\sigma,\sigma,\tau)$ reduces~$L$.
Moreover, all rrs-isotopisms that map $L$ to a reduced square can be uniquely
parameterised by $(i,\delta)$ in this way.  Since $(\alpha^\sigma,\alpha^\sigma,\beta^\tau)$
is an autotopism of $L(\sigma,\sigma,\tau)$, it remains to determine when it is an
automorphism; i.e., when $\alpha^\sigma=\beta^\tau$, equivalently 
when $\alpha^{\xi_i}=\beta$. For any~$j$, 
the triple $(i,j^\alpha,\ell_{ij^\alpha})=(i,j^\alpha,j^{\alpha\xi_i})$
appears in $L$. Also, 
$L$ contains $(i^\alpha,j^\alpha,\ell_{ij}^\beta)=(i^\alpha,j^\alpha,j^{\xi_i\beta})$
since $(\alpha,\alpha,\beta)\in\varGamma$.
However, the condition $\alpha^{\xi_i}=\beta$ implies that
$j^{\alpha\xi_i}=j^{\xi_i\beta}$. Since two triples in a Latin square cannot have
exactly two entries in common, it follows that $i=i^\alpha$.
That is, for $(\alpha^\sigma,\alpha^\sigma,\beta^\tau)$ to be an automorphism of
$L(\sigma,\sigma,\tau)$ it is necessary that $i$ is fixed by~$\alpha$.
Conversely, suppose that $i$ is fixed by~$\alpha$. Since
$(\alpha,\alpha,\beta)\in\varGamma$ we see that $L$ contains the triple
$(i^\alpha,j^\alpha,\ell_{ij}^\beta)=(i,j^\alpha,j^{\xi_i\beta})$ for all
$j$.  From the definition of $\xi_i$ it then follows that
$j^{\alpha\xi_i}=j^{\xi_i\beta}$ for all $j$ and hence $\alpha^{\xi_i}=\beta$,
as required.

By the above argument, for fixed $(\alpha,\alpha,\beta)\in\varGamma$,
there are $(n-1)!\,\lambda(\alpha,\beta)$ choices of $\sigma,\tau$ such that
$(\alpha^\sigma,\alpha^\sigma,\beta^\tau)$ is an automorphism of
$L(\sigma,\sigma,\tau)$, since we have $\lambda(\alpha,\beta)$ choices for $i$
and $(n-1)!$ choices for $\delta$.

However, different choices of $\alpha,\beta,\sigma,\tau$ may give
the same triple
$(L(\sigma,\sigma,\tau),\alpha^\sigma,\beta^\tau)$.  If
$L(\sigma_1,\sigma_1,\tau_1)=L(\sigma_2,\sigma_2,\tau_2)$ for
$\sigma_1,\tau_1,\sigma_2,\tau_2\in\sym_n$, then there is
$(\mu,\mu,\nu)\in\varGamma$ such that $\sigma_2=\mu\sigma_1$
and $\tau_2=\nu\tau_1$.  If we also have $\alpha_2^{\sigma_2}=\alpha_1^{\sigma_1}$
and $\beta_2^{\tau_2}=\beta_1^{\tau_1}$ then
$\alpha_2=\alpha_1^{\mu^{-1}}$ and $\beta_2=\beta_1^{\nu^{-1}}$.  
Furthermore, it follows from 
$(\alpha_1,\alpha_1,\beta_1),(\mu,\mu,\nu)\in\varGamma$ that
$(\alpha_2,\alpha_2,\beta_2)\in\varGamma$.
In summary, each triple $(M,\gamma,\gamma)$ such that $M$ is a reduced square
rrs-isotopic to $L$ and $(\gamma,\gamma,\gamma)$ is an automorphism of~$M$
occurs exactly $|\varGamma|$ times in our counting, once for each choice
of $(\mu,\mu,\nu)$. This completes the proof.
\end{proof}

We next turn to the enumeration of symmetric Latin squares with
additional properties.  \lref{l:obvious}\iref{i:diagsym} shows there
is no further work to do in order to count diagonal symmetric Latin
squares. It also shows that symmetric Latin squares can be idempotent
only for odd orders and can be unipotent only for even
orders. Nevertheless there are connections between the two classes, as
we see shortly.

There is a well known process called \emph{prolongation} (see, for example,
\cite{DKI}) which can be applied to any diagonal Latin square 
$L$ of order $n$ as follows. We remove the triples
\begin{equation}\label{e:before}
  T=\big\{(i,i,L[i,i]):1\le i\le n\big\}
\end{equation}
from $L$ and install in their place the triples
\begin{equation}\label{e:after}
  T^* = \big\{(i,i,n+1):1\le i\le n+1\big\} 
  \cup\big\{(i,n+1,L[i,i]),(n+1,i,L[i,i]):1\le i\le n\big\}
\end{equation}
to create a new Latin square of order $n+1$ that we will denote $L^*$.
It is immediate that $L^*$ is unipotent. Also, it is not hard
to check that $L^*$ is symmetric if and only if $L$ is symmetric.
Analogous claims also hold for semisymmetry and total symmetry.

The reverse process to prolongation, is called either
anti-prolongation or contraction \cite{DKI}.  In this process, if the
triples in \eref{e:after} are present, then we replace them by the
triples in \eref{e:before}, thereby reducing the order of the Latin
square by 1.

For $n\le13$, the first quantity listed in our next result appears in
\Tref{tabsym}, allowing us to infer all the other quantities.

\begin{theorem}\label{t:idemSLS}
Let $n$ be odd. The following objects are equinumerous:
\begin{enumerate}[label={\rm(\roman*)}]\itemsep=0pt
\item\label{i:symsp} species containing symmetric Latin squares of order $n$.
\item\label{i:diagsp} species containing diagonal symmetric Latin
  squares of order $n$,
\item\label{i:idemsp} species containing idempotent symmetric Latin
  squares of order $n$,
\item\label{i:idemisot} isotopism classes containing idempotent
  symmetric Latin squares of order $n$,
\item\label{i:idemrrs} rrs-isotopism classes containing idempotent
  symmetric Latin squares of order $n$,
\item\label{i:idemiso} isomorphism classes of idempotent
  symmetric Latin squares of order $n$,
\item\label{i:unired} isomorphism classes containing unipotent
  symmetric reduced Latin squares of order $n+1$,
\end{enumerate}
\end{theorem}

\begin{proof}
  By \lref{l:obvious}\iref{i:diagsym}, any symmetric Latin square of odd order is
  diagonal, which means that it can be made idempotent by permuting
  the symbols.  The equality between \iref{i:symsp}, \iref{i:diagsp}
  and \iref{i:idemsp} follows.  Also, \lref{l:isotsymm} shows that
  \iref{i:idemsp}, \iref{i:idemisot} and \iref{i:idemrrs} are equal,
  and \cite[Lem.\,6]{WI05} adds \iref{i:idemiso} to that list.

  It thus suffices to show equality between
  \iref{i:idemiso} and \iref{i:unired}. Since we will do this using
  prolongation, it will be more convenient for us to replace
  \iref{i:unired} with a related (and equinumerous) set.
  Let $U_{n+1}$ be the set of symmetric unipotent Latin squares
  of order $n+1$ that have their last row and column in natural order.

Now consider two idempotent symmetric Latin squares $L_1$ and $L_2$ of
order $n$, which we prolong to $L_1^*\in U_{n+1}$ and $L_2^*\in U_{n+1}$.
Suppose there is an isomorphism $I=(\alpha,\alpha,\alpha)$
which maps $L_1$ to $L_2$. We extend the permutation $\alpha\in\sym_n$
to a permutation $\alpha^*\in\sym_{n+1}$ by $(n+1)^\alpha=n+1$ and
define $I^*=(\alpha^*,\alpha^*,\alpha^*)$. Observe that $I$ must fix
$T$ setwise and $I^*$ must fix $T^*$ setwise, where $T$ and $T^*$ are
defined by \eref{e:before} and \eref{e:after}. Since these sets
contain the only triples changed by the prolongations it follows
easily that $I^*$ is an isomorphism from $L_1^*$ to $L_2^*$. We
conclude that the number of isomorphism classes of symmetric
idempotent Latin squares of order $n$ does not exceed the number of
isomorphism classes within $U_{n+1}$.

To show equality we make use of anti-prolongation.
Suppose that we have an isomorphism
$I^*=(\alpha^*,\alpha^*,\alpha^*)$ mapping $L_1^*\in U_{n+1}$ to $L_2^*\in
U_{n+1}$. By definition, $L_1^*[i,i]=L_2^*[i,i]=n+1$ for $1\le i\le n+1$.
Since $I^*$ is an isomorphism it must map the main diagonal of
$L_1^*$ to the main diagonal of $L_2^*$, which requires that
$(n+1)^{\alpha^*}=n+1$. Thus we can define $\alpha\in\sym_n$ as the
restriction of $\alpha^*$ to $\{1,\dots,n\}$, and define
$I=(\alpha,\alpha,\alpha)$. It is now routine to check that $I$ is an
isomorphism between the anti-prolongation $L_1$ of $L_1^*$ and the
anti-prolongation $L_2$ of $L_2^*$. Moreover, both $L_1$ and $L_2$
are symmetric idempotent Latin squares by construction.
\end{proof}

\tref{t:idemSLS} and \Tref{tabsym} combine to tell us the numbers of
isomorphism classes of unipotent symmetric Latin squares as listed in
\Tref{T:unipsym}.  For the rrs-isotopism classes we need the following
result, together with \cite{KO09}:

\begin{table}[htb]
\centering
\def\qd{\kern 0.53em }
\newdimen\digitwidth \setbox0=\hbox{\rm0} \digitwidth=\wd0
\catcode`@=\active \def@{\kern\digitwidth} \centerline{
\vbox{\offinterlineskip \hrule 
\halign{&\vrule#&\strut
        \qd\hfil#\hfil\qd&\vrule#&
        \qd\hfil#\hfil\qd&\vrule#&
        \qd\hfil#\hfil\qd&\vrule#&
        \qd\hfil#\hfil\qd&\vrule#&
        \qd\hfil#\hfil\qd&\vrule#&
        \qd\hfil#\hfil\qd&\vrule#\cr
height2pt&\omit&&&&&&&\cr
&\mybox{order}&&\mybox{isomorphism\\classes}&&\mybox{rrs-isotopism\\classes}&&\mybox{species}& \cr
height2pt&\omit&&&&&&& \cr
\noalign{\hrule height0.8pt}
height2pt&\omit&&&&&&& \cr
&2&&1&&1&&1& \cr
&4&&1&&1&&1& \cr
&6&&1&&1&&1& \cr
&8&&7&&6&&6& \cr
&10&&3460&&396&&396& \cr
&12&&6320290037&&526915620&&526915616& \cr
&14&&15859695832489637513&&1132835421602062347&&1132835421602062347& \cr
height2pt&\omit&&&&&&& \cr
}\hrule}}
\caption{\label{T:unipsym}Counts of symmetric unipotent reduced Latin squares. 
The number of species equals the number of isotopism classes by 
\lref{l:isotsymm}.
}
\end{table}

\begin{theorem}\label{t:uniSLS}
  Let $n$ be even. The following hold:
\begin{enumerate}[label={\rm(\roman*)}]\itemsep=0pt
\item\label{u:specisot} The number of species containing symmetric
  unipotent Latin squares of order $n$ equals the number of isotopism
  classes containing symmetric unipotent Latin squares of order~$n$.
\item\label{u:1fact} The number of rrs-isotopism classes containing
  symmetric unipotent Latin squares of order $n$ equals the number of
  isomorphism classes of $1$-factorisations of~$K_n$.
\item\label{u:2mod4} If $n\equiv2\bmod4$ then the numbers in parts
  \iref{u:specisot} and \iref{u:1fact} above are equal.

\item\label{u:prolong} The number of symmetric unipotent reduced Latin
  squares of order $n$ equals the number of symmetric idempotent Latin
  squares of order $n-1$ which in turn equals the number of symmetric reduced
  Latin squares of order $n-1$.

\item\label{u:prolonger} The number of symmetric unipotent Latin
  squares of order $n$ equals $n!$ times the number of symmetric reduced
  Latin squares of order $n-1$.

\item\label{u:prolongTS} The number of totally symmetric unipotent
  reduced Latin squares of order $n$ equals the number of totally
  symmetric idempotent Latin squares of order $n-1$.
\item\label{u:prolongiso} The number of isomorphism classes of
  symmetric unipotent Latin squares of order $n$ equals the number of
  isomorphism classes of symmetric Latin squares of order $n-1$.
\item\label{u:allhavered} All species, isotopism classes and
  rrs-isotopism classes that contain a symmetric unipotent Latin
  square of order $n$, contain a reduced such square.
\end{enumerate}
\end{theorem}

\begin{proof}
  \lref{l:isotsymm} implies \iref{u:specisot}. Also, \iref{u:1fact} is
  a consequence of a standard encoding of 1-factorisations of complete
  graphs as symmetric unipotent Latin squares (as spelled out
  in~\cite{KMOW14}, for example).

  We next show \iref{u:2mod4}. The case $n=2$ is trivial, so assume
  that $n>2$.  By \lref{l:isotsymm}, it is enough to argue that if
  $\Omega_n$ contains a unipotent symmetric Latin square then
  $n\equiv0\bmod4$. Suppose that $L$ is a unipotent symmetric Latin
  square of order $n$ with an autotopism
  $(\theta,\theta^{-1},\id)$. Let $u$ be the symbol on the main
  diagonal of $L$. For $1\le i\le n$ we know that $L$ contains the
  triple $(i,i,u)$ and $(i^\theta,i^{\theta^{-1}},u)$, from which it
  follows that $\theta=\theta^{-1}$. Hence $\ord(\theta)=2$, and
  \lref{l:class} then tells us that $\theta$ has cycle structure
  $2^{n/2}$. The combined action of transposition and the autotopism
  $(\theta,\theta,\id)$ forms $n$ orbits of size 2 on triples, and all
  other orbits have size 4.  The symbol $u$ occupies $n/2$ of the
  orbits of size 2. Since $n-n/2<n-1$, there must be some symbol $s$
  that does not appear in any of the orbits of size 2. As $s$ must
  occur $n$ times in $L$, it follows that $n$ must be a multiple of
  $4$, completing the proof of \iref{u:2mod4}.

  We next justify \iref{u:prolong}. Prolongation is a bijection
  between symmetric idempotent Latin squares of order $n-1$ and
  symmetric unipotent Latin squares of order $n$ that have their
  \emph{last} row and column in natural order. It is clear that the
  latter set is equinumerous with the symmetric unipotent reduced
  Latin squares of order $n$. 

  For any symmetric idempotent Latin square of order $n-1$
  there is a unique symbol permutation that maps it to a
  symmetric reduced Latin square, namely the permutation which
  puts the first row in order. Conversely,
  for any symmetric reduced Latin square of order $n-1$
  there is a unique symbol permutation that maps it to a
  symmetric idempotent Latin square, namely the permutation which
  puts the main diagonal in order (note that $n-1$ is odd, and we have
  \lref{l:obvious}\iref{i:diagsym}). It follows that among the
  symmetric Latin squares of order $n-1$ the number of idempotent squares
  equals the number of reduced squares.
  These observations combine to prove \iref{u:prolong}.

  Symbol permutations provide an $n!$ to $1$ map from symmetric
  unipotent Latin squares of order $n$ to symmetric unipotent reduced
  Latin squares of order $n$. Hence \iref{u:prolonger} follows from
  \iref{u:prolong}.  

  The proof of \iref{u:prolongTS} is identical to the proof of the
  first claim in \iref{u:prolong}.

  For \iref{u:prolongiso} it helps to observe that
  isomorphisms can be used to change the symbol on the main diagonal
  of a unipotent Latin square into the symbol $n$. After that, the argument is
  identical to the proof that \tref{t:idemSLS}\iref{i:idemiso} and
  \tref{t:idemSLS}\iref{i:unired} are equal.
  
  To prove \iref{u:allhavered} we apply \lref{l:obvious}\iref{i:obvredsym}
  and note that symbol permutations preserve unipotence.
\end{proof}

We remark that \tref{t:uniSLS}\iref{u:allhavered} shows that the
counting problems for species, isotopism classes and rrs-isotopism
classes that contain a symmetric unipotent Latin square are unchanged
by adding an extra condition that the squares should be
reduced. However, this is not true for isomorphism classes, as can be
seen by comparing \Tref{T:comloop} to \Tref{T:unipsym}, in light of
\tref{t:uniSLS}\iref{u:prolongiso}.

It only remains to explain how we deduced the numbers of species in
\Tref{T:unipsym}. As seen in the proof of
\tref{t:uniSLS}\iref{u:2mod4}, it suffices to find the difference
between the number of species and number of rrs-isotopism classes of
unipotent symmetric Latin squares with a principal autotopism
$(\theta,\theta,\id)$ where $\theta$ has cycle type $2^{n/2}$.  We had
already generated rrs-isotopism class representatives of such squares
in the process of compiling \Tref{T:agopts}. We know from
\tref{t:uniSLS}\iref{u:2mod4} that there will be equal numbers of
species and rrs-isotopism classes when $n\equiv2\bmod4$, but we found
this was also true when $n\in\{4,8\}$. For order 12 there was a small
difference, with 4851 rrs-isotopism classes but only 4847 species of
unipotent Latin squares in $\Omega_{12}$. Thus there are 4 fewer
species than rrs-isotopism classes overall, among the unipotent
symmetric Latin squares of order 12. We finish the section by giving
an example demonstrating this phenomenon.  Let $A_{12}$ be the
unipotent symmetric Latin square below and let $B_{12}$ be the square
that is obtained from $A_{12}$ by replacing each of the four
highlighted subsquares by the other possible subsquare on the same
symbols. 
\begin{equation*}
\left[\begin{array}{cccccccccccc}
 1& 2& 3& 4& 5& 6& 7& 8& 9&10&11&12\\
 2& 1& 4& 3& 6& 5& 8& 7&10& 9&12&11\\
 3& 4& 1& 2& 7& 9& 5&11& 8&12& 6&10\\
 4& 3& 2& 1& 9& 7&11& 5&12& 8&10& 6\\
 5& 6& 7& 9& 1& 2&10&12& 3&11& 4& 8\\
 6& 5& 9& 7& 2& 1&12&10&11& 3& 8& 4\\
 7& 8& 5&11&10&12& 1& 2& \mk6& \mk4& 3& 9\\
 8& 7&11& 5&12&10& 2& 1& \mk4& \mk6& 9& 3\\
 9&10& 8&12& 3&11& \mk6& \mk4& 1& 2& \mk5& \mk7\\
10& 9&12& 8&11& 3& \mk4& \mk6& 2& 1& \mk7& \mk5\\
11&12& 6&10& 4& 8& 3& 9& \mk5& \mk7& 1& 2\\
12&11&10& 6& 8& 4& 9& 3& \mk7& \mk5& 2& 1\\
\end{array}\right]
\end{equation*}
Then $A_{12}$ is isotopic to $B_{12}$ using the isotopism
$(\alpha,\gamma,\gamma)$ where $\alpha=(3,10,7,6,11,4,9,8,5,12)$ and
$\gamma=(1,2)(3,9,7,5,11)(4,10,8,6,12)$, but $A_{12}$ is not
rrs-isotopic to $B_{12}$.

\section{Semisymmetric Latin squares}\label{s:semisym}

In this section we report on the results of our enumeration of
semisymmetric Latin squares of small order up to various notions of
equivalence, and with or without certain extra properties. We state
the results first, and then at the end of the section we offer some
discussion of how they were obtained.  We start with \Tref{tabsemi},
which gives the counts of semisymmetric Latin squares (without any
additional restrictions) of orders up to 11.

\begin{table}[hbt]
\def\qd{\quad}
\newdimen\digitwidth \setbox0=\hbox{\rm0} \digitwidth=\wd0
\catcode`@=\active \def@{\kern\digitwidth} \centerline{
\vbox{\offinterlineskip \hrule 
\halign{&\vrule#&\strut
        \qd\hfil#\hfil\qd&\vrule#&
        \qd\hfil#\hfil\qd&\vrule#&
        \qd\hfil#\hfil\qd&\vrule#&
        \qd\hfil#\hfil\qd&\vrule#&
        \qd\hfil#\hfil\qd&\vrule#&
        \qd\hfil#\hfil\qd&\vrule#\cr
height2pt&\omit&&&&&&&&&\cr
&\mybox{order}&&\mybox{isomorphism\\classes}&&\mybox{isotopism\\classes}&&\mybox{species}&&\mybox{all squares}&\cr
height2pt&\omit&&&&&&&&&\cr
\noalign{\hrule height0.8pt}
height2pt&\omit&&&&&&&&&\cr
&2&&1&&1&&1&&2&\cr
&3&&2&&1&&1&&3&\cr
&4&&3&&2&&2&&18&\cr
&5&&4&&2&&2&&120&\cr
&6&&9&&7&&7&&2880&\cr
&7&&41&&33&&28&&140256&\cr
&8&&595&&557&&366&&20782080&\cr
&9&&26620&&26511&&13899&&9569532672&\cr
&10&&3908953&&3908091&&1968997&&14175610675200&\cr
&11&&1867918845&&1867909542&&934327507&&74559788174868480&\cr
height2pt&\omit&&&&&&&&&\cr
}\hrule}}
\caption{\label{tabsemi}Semisymmetric Latin squares}
\end{table}

The isomorphism classes of semisymmetric squares for orders up to 6
were listed by Sade~\cite{Sad22}, who in the same paper gave the
species of semisymmetric squares of order~7.  Later, the same
author \cite{Sad25} listed the isomorphism classes of semisymmetric
squares of order~7. In all these cases, \Tref{tabsemi} agrees
with Sade's results.
As well as enumerating classes of general semisymmetric squares
Sade also noted how many of these squares were diagonal, idempotent
or unipotent. Motivated by his work, we also count these classes.

All triples have an orbit of length 3 under $(2,3,1)$-conjugation
except for the constant triples $(i,i,i)$ that arise from idempotent
elements. Hence the following result is an easy consequence of
counting triples.

\begin{theorem}\label{t:neccon}
  In a semisymmetric Latin square of order $n$ the number of
  idempotent elements is congruent to $n^2\bmod3$. Hence, for there to
  exist an idempotent semisymmetric Latin square of order $n$ it is
  necessary that $n\not\equiv2\bmod3$.  For there to exist a unipotent
  semisymmetric Latin square of order $n$ it is necessary that
  $n\not\equiv0\bmod3$.
\end{theorem}

\tref{t:neccon} was noted in \cite{BBW09}, which established that
idempotent semisymmetric Latin squares exist for all orders
$n\not\equiv2\bmod3$ except $n=6$.  Also unipotent semisymmetric Latin
squares exist for all orders $n\not\equiv0\bmod3$ except $n=7$.
These results are consistent with what we found in our work.

The following three tables omit orders which are immediately
eliminated by \tref{t:neccon}. For other orders up to order 12 we give
the numbers of semisymmetric squares that are, respectively, diagonal
in \Tref{T:diag} and idempotent in \Tref{T:idem}.  Similarly,
Table~\ref{T:semisymloop} gives counts up to order 13 of semisymmetric
reduced Latin squares. 

\begin{lemma}\label{l:semiuniloop}
  The number of unipotent semisymmetric Latin squares of order $n$ is
  $n$ times the number of reduced semisymmetric Latin squares of order
  $n$. The numbers of isomorphism classes, isotopism classes and
  species containing unipotent semisymmetric Latin squares of order
  $n$ are equal, respectively, to the numbers of isomorphism classes,
  isotopism classes and species containing reduced semisymmetric Latin
  squares of order $n$.  The above statements also hold with
  ``semisymmetric'' replaced throughout by ``totally symmetric''.
\end{lemma}

\begin{proof}
Note that the presence of the triple $(i,x,x)$
implies the presence of the triple $(x,x,i)$, for each $x$, and
conversely.  Hence, semisymmetric quasigroups are unipotent if and
only if they are loops.  Table~\ref{T:semisymloop} thus also counts
unipotent semisymmetric Latin squares (semisymmetric loops), with one
caveat about the last column in the table. There are the same number
of semisymmetric loops with identity element $i$ as there are with
identity element $1$ (an isomorphism interchanging $i$ with $1$ maps
one set to the other), regardless of what value $i$ has.
Hence the total number of unipotent
semisymmetric Latin squares of order $n$ is $n$ times the number of
reduced semisymmetric Latin squares of the same order.

The above argument applies without change to totally symmetric Latin
squares.
\end{proof}

\begin{table}[hbt]
\def\qd{\quad}
\newdimen\digitwidth \setbox0=\hbox{\rm0} \digitwidth=\wd0
\catcode`@=\active \def@{\kern\digitwidth} \centerline{
\vbox{\offinterlineskip \hrule 
\halign{&\vrule#&\strut
        \qd\hfil#\hfil\qd&\vrule#&
        \qd\hfil#\hfil\qd&\vrule#&
        \qd\hfil#\hfil\qd&\vrule#&
        \qd\hfil#\hfil\qd&\vrule#&
        \qd\hfil#\hfil\qd&\vrule#\cr
        height2pt&\omit&&&&&&&&&\cr
&\mybox{order}&&\mybox{isomorphism\\classes}&&\mybox{isotopism\\classes}&&\mybox{species}&&\mybox{all squares}&\cr        
height2pt&\omit&&&&&&&&&\cr
\noalign{\hrule height0.8pt}
height2pt&\omit&&&&&&&&&\cr
&2&&0&&0&&0&&0&\cr
&3&&2&&1&&1&&3&\cr
&4&&1&&1&&1&&2&\cr
&5&&1&&1&&1&&30&\cr
&6&&0&&0&&0&&0&\cr
&7&&7&&5&&5&&3000&\cr
&8&&2&&2&&2&&20160&\cr
&9&&112&&91&&76&&19571328&\cr
&10&&2369&&2341&&1285&&8136806400&\cr
&11&&347299&&347299&&175105&&13826847640320&\cr
&12&&237570420&&237569195&&118815560&&113788019281305600&\cr
height2pt&\omit&&&&&&&&&\cr
}\hrule}}
\caption{\label{T:diag}Diagonal semisymmetric Latin squares.}
\medskip
\end{table}

On examining \Tref{T:diag} we were drawn to wonder whether
isotopic diagonal semisymmetric Latin squares of order $n\equiv 2\bmod3$
are necessarily isomorphic. It turns out that the answer is negative,
as can be seen from the following counterexample of order 35.
Consider the direct products $A_1\times B$ and $A_2\times B$ where
\[
A_1=\left[\begin{array}{ccccccc}
1&3&2&5&4&7&6\\
3&4&1&2&6&5&7\\
2&1&5&7&3&6&4\\
5&2&7&6&1&4&3\\
4&6&3&1&7&2&5\\
7&5&6&4&2&3&1\\
6&7&4&3&5&1&2\\
\end{array}\right],
\qquad
A_2=\left[\begin{array}{ccccccc}
1&7&6&2&3&4&5\\
4&2&5&3&6&7&1\\
5&4&3&7&2&1&6\\
6&1&2&4&7&5&3\\
7&3&1&6&5&2&4\\
3&5&7&1&4&6&2\\
2&6&4&5&1&3&7\\
\end{array}\right]
\]
and $B$ has order $5$ and is defined by $B_{ij}\equiv-i-j\bmod5$. By
inspection, $A_1$, $A_2$ and $B$ are diagonal semisymmetric Latin
squares. Also, applying the isotopism
$\big((256)(347),(265)(374),\id\big)$ to $A_1$, we find that it is
isotopic to $A_2$ (in fact it can be shown that both squares are
isotopic to the Cayley table of $\Z_7$). However, $A_1$ and $A_2$ are
not isomorphic, since they have different numbers of idempotent
elements ($A_1$ has one idempotent element, while $A_2$ has seven).  It
is simple to check that $A_1\times B$ and $A_2\times B$ inherit the
properties of being isotopic diagonal semisymmetric Latin
squares. They are not isomorphic, since they have different numbers of
idempotent elements.

\begin{table}[hbt]
\def\qd{\quad}
\newdimen\digitwidth \setbox0=\hbox{\rm0} \digitwidth=\wd0
\catcode`@=\active \def@{\kern\digitwidth} \centerline{
\vbox{\offinterlineskip \hrule 
\halign{&\vrule#&\strut
        \qd\hfil#\hfil\qd&\vrule#&
        \qd\hfil#\hfil\qd&\vrule#&
        \qd\hfil#\hfil\qd&\vrule#&
        \qd\hfil#\hfil\qd&\vrule#&
        \qd\hfil#\hfil\qd&\vrule#\cr
height2pt&\omit&&&&&&&&&\cr
&\mybox{order}&&\mybox{isomorphism\\classes}&&\mybox{isotopism\\classes}&&\mybox{species}&&\mybox{all squares}&\cr
height2pt&\omit&&&&&&&&&\cr
\noalign{\hrule height0.8pt}
height2pt&\omit&&&&&&&&&\cr
&3&&1&&1&&1&&1&\cr
&4&&1&&1&&1&&2&\cr
&6&&0&&0&&0&&0&\cr
&7&&4&&3&&3&&480&\cr
&9&&20&&19&&17&&2274048&\cr
&10&&241&&238&&141&&757555200&\cr
&12&&9801188&&9801140&&4905666&&4693077997977600&\cr
height2pt&\omit&&&&&&&&&\cr
}\hrule}}
\caption{\label{T:idem}Idempotent semisymmetric Latin squares.}
\end{table}

\begin{table}[hbt]
\def\qd{\quad}
\newdimen\digitwidth \setbox0=\hbox{\rm0} \digitwidth=\wd0
\catcode`@=\active \def@{\kern\digitwidth} \centerline{
\vbox{\offinterlineskip \hrule 
\halign{&\vrule#&\strut
        \qd\hfil#\hfil\qd&\vrule#&
        \qd\hfil#\hfil\qd&\vrule#&
        \qd\hfil#\hfil\qd&\vrule#&
        \qd\hfil#\hfil\qd&\vrule#&
        \qd\hfil#\hfil\qd&\vrule#\cr
height2pt&\omit&&&&&&&&&\cr
&\mybox{order}&&\mybox{isomorphism\\classes}&&\mybox{isotopism\\classes}&&\mybox{species}&&\mybox{all squares}&\cr
height2pt&\omit&&&&&&&&&\cr
\noalign{\hrule height0.8pt}
height2pt&\omit&&&&&&&&&\cr
&2&&1&&1&&1&&1&\cr
&4&&1&&1&&1&&1&\cr
&5&&1&&1&&1&&2&\cr
&7&&0&&0&&0&&0&\cr
&8&&4&&4&&3&&480&\cr
&10&&20&&20&&18&&2274048&\cr
&11&&241&&241&&143&&757555200&\cr
&13&&9801188&&9801188&&4905693&&4693077997977600&\cr
height2pt&\omit&&&&&&&&&\cr
}\hrule}}
\caption{\label{T:semisymloop}Reduced semisymmetric Latin squares.}
\end{table}

The following special case of \cite[Thm~1]{Art59}
explains why the ``isomorphism classes'' column matches
the ``isotopism classes'' column in \Tref{T:semisymloop}.

\begin{theorem}\label{t:Artzy}
Any two isotopic semisymmetric loops are isomorphic.
\end{theorem}

Another pattern that became evident when we compiled
\Tref{T:idem} and \Tref{T:semisymloop} is the following.

\begin{theorem}\label{t:prolong}
The number of isomorphism classes of semisymmetric idempotent Latin
squares of order $n$ equals the number of isomorphism classes of semisymmetric
unipotent Latin squares of order $n+1$.
\end{theorem}

\begin{proof}
  Consider two idempotent semisymmetric Latin squares $L_1$ and $L_2$
  of order $n$, which we prolong to unipotent semisymmetric Latin
  squares $L_1^*$ and $L_2^*$. Any isomorphism which maps $L_1$ to
  $L_2$ can be extended to an isomorphism which maps $L_1^*$ to $L_2^*$,
  for exactly the same reasons as we saw in the proof of \tref{t:idemSLS}.
  
To work in the other direction, we will make use of anti-prolongation.
First, we argue that every isomorphism class $C$
of semisymmetric unipotent squares of order $n+1$ has a representative
containing the triples
\begin{equation}\label{e:stndunip}
  \big\{(i,i,n+1),(i,n+1,i),(n+1,i,i):1\le i\le n+1\big\}.
\end{equation}
Suppose that $L\in C$ and that $L$
has the symbol $u$ in every position on its main diagonal. Let $\tau$
be the transposition $(u,n+1)\in\sym_{n+1}$. Then by applying the
isomorphism $(\tau,\tau,\tau)$ to $L$ we get a (necessarily semisymmetric)
square $L'$ containing the triples \eref{e:stndunip}.

Suppose that we have an isomorphism
$I^*=(\alpha^*,\alpha^*,\alpha^*)$ mapping $L_1^*\in C$ to $L_2^*\in
C$. By the above argument, we may assume that both squares contain the
triples in \eref{e:stndunip}. Since $I^*$ is an isomorphism it must
map the main diagonal of $L_1^*$ to the main diagonal of $L_2^*$,
which requires that $(n+1)^{\alpha^*}=n+1$. Thus we can define
$\alpha\in\sym_n$ as the restriction of $\alpha^*$ to $\{1,\dots,n\}$,
and define $I=(\alpha,\alpha,\alpha)$. It is now routine to check that
$I$ is an isomorphism between the anti-prolongation $L_1$ of $L_1^*$
and the anti-prolongation $L_2$ of $L_2^*$. Moreover, both $L_1$ and
$L_2$ will be symmetric idempotent Latin squares. The theorem follows.
\end{proof}

It is important to note that \tref{t:prolong} does not generalise to
isotopism classes or species. Consider the following pair of
semisymmetric idempotent Latin squares of order 9.
$$\left[\begin{array}{ccccccccc}
1&3&2&7&8&9&4&5&6\\
3&2&1&8&9&7&6&4&5\\
2&1&3&9&7&8&5&6&4\\
7&8&9&4&6&5&1&2&3\\
8&9&7&6&5&4&3&1&2\\
9&7&8&5&4&6&2&3&1\\
4&6&5&1&3&2&7&9&8\\
5&4&6&2&1&3&9&8&7\\
6&5&4&3&2&1&8&7&9\\
\end{array}\right]\quad\quad
\left[\begin{array}{ccccccccc}
1&3&2&7&8&9&5&6&4\\ 
3&2&1&8&9&7&4&5&6\\ 
2&1&3&9&7&8&6&4&5\\
9&7&8&4&6&5&1&2&3\\ 
7&8&9&6&5&4&3&1&2\\ 
8&9&7&5&4&6&2&3&1\\
4&6&5&2&1&3&7&9&8\\ 
5&4&6&3&2&1&9&8&7\\ 
6&5&4&1&3&2&8&7&9\\ 
\end{array}\right]$$ 
Applying the isotopism $\big((789),(465),(456)(798)\big)$
to the left hand square, we learn that the two squares are isotopic
(in fact, it can be shown that both squares are isotopic to the Cayley
table of $\Z_3\times\Z_3$). However, the prolongations of these squares to
semisymmetric unipotent squares of order 10 do not even belong to the
same species, as can easily be established by counting their
subsquares or transversals.

To finish the section we briefly describe how we obtained the data
that we have given in Tables \ref{tabsemi} to \ref{T:semisymloop}.  For
each table the first task was to find a list of representatives
of the isomorphism classes. By filtering such a list it is a simple
matter to count isotopism classes and species. Also the number of ``all
squares'' with the relevant properties can be calculated using the
orbit-stabiliser theorem, by finding the order of the automorphism
group for each isomorphism class representative. Hence the main challenge
was finding the isomorphism class representatives. To do this, we first
filled the main diagonal in all possible ways (up to isomorphism) that
were consistent with \tref{t:neccon}. Of course, for \Tref{T:idem}
and \Tref{T:semisymloop} the main diagonal was completely determined, whilst
there were more choices in the other cases. It was also helpful that for
each of the classes that we generated, membership of the class was determined
by the main diagonal, together with semisymmetry. Hence the generation
for each class was the same after the initial step. Indeed,
some of the classes that we generated are obviously subclasses of others,
and proceeding as we did allowed us to reuse results from the subclasses
without repeating the work.

After filling the main diagonal we proceeded row by row in backtrack
fashion.  Any time that a triple was added we also added every triple
which followed from it by semisymmetry.  Our two independent
computations did isomorphism screening at different points, but both
screened after filling the main diagonal and after filling the whole
square (and at some points in between).

The hardest case was the generation of unrestricted semisymmetric
Latin squares of order 11, which took roughly 4 days of computation.
There were 4586 initial options for the main diagonal. Of those, 117
turned out to have no completion to a semisymmetric Latin square. The
most productive diagonal could be completed in 35046912 non-isomorphic
ways.

\section{Correction to Sade}\label{s:Sade}

In this section we demonstrate an error which our enumerations have
uncovered in the pioneering works of Sade on semisymmetric Latin
squares. In his terminology, a \emph{left autotopism} is an isotopism from
a Latin square to its $(3,1,2)$-conjugate.  (Note that Sade described
the $(3,1,2)$-conjugate of a Latin square as its ``transpose'', but we
shall avoid this confusing name since in ordinary matrix terminology
the transpose is the $(2,1,3)$-conjugate. Also, it should be noted
that Sade expressed all his results in terms of quasigroups, but we
describe them in terms of Latin squares.) We say that a square
has a \emph{semisymmetric form} if some member of its species is
semisymmetric.  Sade observed rightly that a square may possess a left
autotopism without the square having a semisymmetric form. His error
occurred when he tried to determine the smallest order for which this
happens.

\begin{theorem}\label{t:Sade}
A necessary and sufficient condition that a Latin square $L$ has a
semisymmetric form is that $L$ possesses a left autotopism
$(\alpha,\beta,\gamma)$ such that the permutation $\gamma\beta\alpha$
has order not divisible by $3$.
\end{theorem}

Sade proved the above theorem in \cite{Sad21}, where he also gave an
example that he claimed was of the lowest possible
order for a square possessing a left autotopism but having no semisymmetric
form. This example, which has order 10, was reproduced in
\cite[p.~63]{DKI}.  However, as we shall now point out, Sade's claim
is false since there are 11 species of Latin squares of order 9
which have the desired properties.  One example, $L$, of such a Latin
square is given in \eref{e:SadeCE}. It is isotopic to its
$(3,1,2)$-conjugate by applying the permutation $\tau=(456)(789)$ to
its symbols. It also has an automorphism $(\tau,\tau,\tau)$.  These
two symmetries generate an autoparatopism group of order 9. Clearly
then, the three left autotopisms which $L$ possesses are all of order 3
so, by \tref{t:Sade}, $L$ has no semisymmetric form.

\begin{equation}
\left[\begin{array}{ccccccccc}
2&1&3&5&6&4&9&7&8\\
1&3&2&6&4&5&8&9&7\\
3&2&1&7&8&9&5&6&4\\
5&4&8&2&7&1&6&3&9\\
6&5&9&1&2&8&7&4&3\\
4&6&7&9&1&2&3&8&5\\
7&8&6&4&9&3&1&5&2\\
8&9&4&3&5&7&2&1&6\\
9&7&5&8&3&6&4&2&1\\
\end{array}\right]
\label{e:SadeCE}
\end{equation}

Sade's error surfaced when we compared our count of 13899 species
containing semisymmetric squares of order 9 to the number of main
classes of squares with left autotopisms. This latter number, which was
found as part of the enumeration in \cite{MMM07}, turns out to be
13910 and the 11 extra species must all be counterexamples to Sade's
claim. By comparing the corresponding numbers for smaller orders, we
can be sure that order 9 is the smallest order for which a square
without a semisymmetric form can have a left autotopism. Note that
Kolesova, Lam and Thiel \cite{KLT90} established that there are 366
species and 557 isotopism classes of squares of order 8 which possess a
left autotopism.  These numbers match the corresponding numbers of
semisymmetric squares as given in \Tref{tabsemi}.  Sade's enumeration
in \cite{Sad22} is sufficient to handle the smaller orders. Our
discovery of Sade's error was communicated to Keedwell in time for it
to be noted in \cite{KD15}.

\section{Totally symmetric squares}\label{s:totsym}

In this section we report on the results of our enumeration of totally
symmetric Latin squares of small order. Again, we save discussion of
algorithmic details until after we have given the results.  The counts
of totally symmetric squares for orders up to 13 are shown in
\Tref{tabtotsym}. Note that these results confirm and extend the
results of Bailey, Preece and Zemroch \cite{Bai1}. They listed the
isomorphism classes and species of totally symmetric squares for
orders up to $7$, and calculated the total number of such squares of
these orders.  Later, in \cite{Bai2} and \cite{Bai3} Bailey extended
these results up to order ten, although she did not count the species
for the new orders.

\begin{table}[hbt]
\def\qd{\quad}
\newdimen\digitwidth \setbox0=\hbox{\rm0} \digitwidth=\wd0
\catcode`@=\active \def@{\kern\digitwidth} \centerline{
\vbox{\offinterlineskip \hrule 
\halign{&\vrule#&\strut
        \qd\hfil#\hfil\qd&\vrule#&
        \qd\hfil#\hfil\qd&\vrule#&
        \qd\hfil#\hfil\qd&\vrule#&
        \qd\hfil#\hfil\qd&\vrule#\cr
height2pt&\omit&&&&&&&\cr
&\mybox{order}&&\mybox{isomorphism\\classes}&&\mybox{species}&&\mybox{all squares}&\cr
height2pt&\omit&&&&&&&\cr
\noalign{\hrule height0.8pt}
height2pt&\omit&&&&&&&\cr
&2&&1&&1&&2&\cr
&3&&2&&1&&3&\cr
&4&&2&&2&&16&\cr
&5&&1&&1&&30&\cr
&6&&3&&2&&480&\cr
&7&&3&&3&&1290&\cr
&8&&13&&13&&163200&\cr
&9&&12&&8&&471240&\cr
&10&&139&&139&&386400000&\cr
&11&&65&&65&&2269270080&\cr
&12&&25894&&25888&&12238171545600&\cr
&13&&24316&&24316&&149648961369600&\cr
&14&&92798256&&92798256&&8089070513113497600&\cr
&15&&122859802&&122859796&&160650421233958656000&\cr
height2pt&\omit&&&&&&&\cr
}\hrule}}
\caption{\label{tabtotsym}Totally symmetric Latin squares.}
\end{table}

A similar result to the next lemma was proved in \cite{WI05} in the
odd order case.

\begin{lemma}\label{l:princautabel}
Let $P$ be the group of principal autotopisms of a symmetric Latin
square $L$ of order $n$.  Then $P$ is abelian. Also, if
$(\alpha,\beta,\id)\in P$ then $\alpha=\beta^{-1}$ and $\ord(\alpha)$
divides $n$.
\end{lemma}

\begin{proof}
Suppose $(\alpha,\beta,\id)\in P$.  For $1\le i\le n$ define
$k_i$ to be the symbol such that $L$ contains the triple
$(i,i^\alpha,k_i)$. Since $(\alpha,\beta,\id)$ is an autotopism of $L$ we
have that $(i^\alpha,i^{\alpha\beta},k_i)\in L$, and since $L$ is symmetric
we know that $(i^\alpha,i,k_i)\in L$. Any two triples of $L$ that agree
in two coordinates must be equal, so $i^{\alpha\beta}=i$. As $i$ was
arbitrary, we see that $\alpha\beta=\id$, so $\beta=\alpha^{-1}$.

Next suppose that $(\alpha_1,\alpha_1^{-1},\id)$ and 
$(\alpha_2,\alpha_2^{-1},\id)$ are two arbitrary elements of $P$. By 
composition we know that
$(\alpha_1\alpha_2,\alpha_1^{-1}\alpha_2^{-1},\id)\in P$. Hence,
by the characterisation just shown, 
$\alpha_1\alpha_2=(\alpha_1^{-1}\alpha_2^{-1})^{-1}=\alpha_2\alpha_1$.
It follows that $P$ is abelian as claimed.

Finally, we consider the order of $\alpha$ for $(\alpha,\alpha^{-1},\id)\in P$.
Let $c$ be the length of the shortest cycle in the
cycle decomposition of $\alpha$. Then $\alpha^c$ will have fixed points
and $(\alpha^c,\alpha^{-c},\id)\in P$. So by \lref{l:class},
it follows that $\alpha^c=\id$. This means that all cycles of $\alpha$ have
length $c$. So $\ord(\alpha)=c$ and $c$ divides $n$, completing the proof.
\end{proof}

Our next result shows why in \Tref{tabtotsym} the number of
isomorphism classes always matches the number of species for orders
that are not divisible by $3$.

\begin{theorem}\label{t:isotisomtotsym}
Suppose $A$ and $B$ are totally symmetric Latin squares of order
$n\not\equiv0\bmod3$.  If $A$ and $B$ are paratopic then $A$ and $B$ are
isomorphic.
\end{theorem}

\begin{proof}
Suppose $A$ is paratopic to $B$. Then $A$ is isotopic
to $B$, by \lref{l:TSspeciso}.
Moreover, by replacing $B$ by an isomorph of $B$ if necessary,
we may assume that $A=B(\alpha,\beta,\id)$. From total symmetry it
then follows that 
\begin{equation}\label{e:AB1}
A=B(\alpha,\beta,\id)=B(\beta,\alpha,\id)=B(\alpha,\id,\beta)=B(\id,\alpha,\beta)=B(\beta,\id,\alpha)=B(\id,\beta,\alpha).
\end{equation}
These relationships will be used repeatedly in what follows.
For starters, we have
\[
A=B(\alpha,\id,\beta)=A(\id,\alpha^{-1},\beta^{-1})(\alpha,\id,\beta)
=A(\alpha,\alpha^{-1},\id)
\] 
and a similar argument shows that
$(\beta,\beta^{-1},\id)$ is an autotopism of $A$.
Thus, by \lref{l:princautabel}, we
see that $\alpha$ and $\beta$ commute and that $\ord(\alpha)$ and
$\ord(\beta)$ both divide $n$. In particular,
$\ord(\alpha)\not\equiv0\bmod3$ and $\ord(\beta)\not\equiv0\bmod3$. Now
by \eref{e:AB1},
\begin{equation*}\label{e:AB2}
A=B(\alpha,\beta,\id)(\id,\beta^{-1},\alpha^{-1})(\alpha,\beta,\id)
(\alpha^{-1},\id,\beta^{-1})(\beta,\alpha,\id)=B(\alpha\beta,\alpha\beta,
(\alpha\beta)^{-1})
\end{equation*}
and
\[
B=B(\alpha,\beta,\id)(\id,\alpha^{-1},\beta^{-1})(\beta,\alpha,\id)
(\alpha^{-1},\id,\beta^{-1})=B(\beta,\beta,\beta^{-2}).
\]
Similarly $B=B(\alpha,\alpha,\alpha^{-2})$. So for any $k\in\Z$,
\begin{equation}\label{e:AB3}
A=B(\alpha,\alpha,\alpha^{-2})^k(\beta,\beta,\beta^{-2})^k
(\alpha\beta,\alpha\beta,(\alpha\beta)^{-1})
=B\big((\alpha\beta)^{k+1},(\alpha\beta)^{k+1},(\alpha\beta)^{-2k-1}\big)
\end{equation}
Now $\ord(\alpha\beta)$ divides the least common
multiple of $\ord(\alpha)$ and $\ord(\beta)$. In particular,
$\ord(\alpha\beta)\not\equiv0\bmod3$, so there exists $k\in\Z$
such that $k+1\equiv-2k-1\bmod\ord(\alpha\beta)$. Thus \eref{e:AB3} shows that
$A$ is isomorphic to $B$, as required.
\end{proof}

\Tref{tabtotsym} shows that the requirement for $n\not\equiv0\bmod3$ cannot
be abandoned in \tref{t:isotisomtotsym}. In fact we have:

\begin{lemma}
For any order $n\equiv0\bmod3$ there exist isotopic totally symmetric
Latin squares that are not isomorphic.
\end{lemma}

\begin{proof}
Define Latin squares $A,B$ on symbols $\Z_n$ by
$A[i,j]\equiv-i-j\bmod n$ and $B[i,j]\equiv1-i-j\bmod n$. It is easy to
check that both $A$ and $B$ are totally symmetric.
Clearly $A$ is isotopic to $B$ via the symbol permutation $x\mapsto
x+1\bmod n$.  However $A$ and $B$ are not isomorphic. To see this, note
that $A$ has 3 idempotent elements because
$A[cn/3,cn/3]\equiv-2cn/3\equiv cn/3\bmod n$ for
$c\in\{0,1,2\}$. However, $B$ has no idempotent elements since 
$B[i,i]=1-2i\not\equiv i\bmod n$ for all $i\in\Z_n$. 
\end{proof}

Our final result discusses idempotent totally symmetric Latin squares,
and the different guises in which they appear.

\begin{theorem}\label{t:STS}
Let $n$ be given. The following objects are equinumerous:
\begin{enumerate}[label={\rm(\roman*)}]\itemsep=0pt
\item\label{i:TSSTS} Isomorphism classes of Steiner triple systems on
  $n$ points,
\item\label{i:TSidemisom} Isomorphism classes of idempotent totally
  symmetric Latin squares of order $n$,
\item\label{i:TSidemisot} Isotopism classes containing idempotent totally
  symmetric Latin squares of order $n$,
\item\label{i:TSidemspec} Species containing idempotent totally symmetric
  Latin squares of order $n$.
\item\label{i:TSloop} Isomorphism classes of totally symmetric loops 
  of order $n+1$,
\item\label{i:TSuniisom} Isomorphism classes of totally symmetric unipotent
  Latin squares of order $n+1$,
\item\label{i:TSredisom} Isomorphism classes containing totally
  symmetric reduced Latin squares of order $n+1$,
\item\label{i:TSuniisot} Isotopism classes containing totally
  symmetric unipotent Latin squares of order $n+1$,
\item\label{i:TSredisot} Isotopism classes containing totally
  symmetric reduced Latin squares of order $n+1$,
\item\label{i:TSunispec} Species containing totally symmetric unipotent
  Latin squares of order $n+1$,
\item\label{i:TSredspec} Species containing totally symmetric reduced
  Latin squares of order $n+1$,
\end{enumerate}
\end{theorem}

\begin{proof} 
The correspondence between Steiner triple systems on $n$ points and
idempotent totally symmetric Latin squares of order $n$ is well
known (see, for example, \cite[Thm~2.2.3]{DKI}). Hence
$\iref{i:TSSTS}=\iref{i:TSidemisom}$. To see that
$\iref{i:TSidemisom}=\iref{i:TSuniisom}$ we use the same argument
used to prove \tref{t:prolong}.
Also, \cite[Lem.\,6]{WI05} shows
that $\iref{i:TSidemisom}=\iref{i:TSidemisot}$ and \tref{t:Artzy} shows that
$\iref{i:TSuniisom}=\iref{i:TSuniisot}$ and $\iref{i:TSredisom}=\iref{i:TSredisot}$.
\lref{l:TSspeciso} shows that
$\iref{i:TSidemisot}=\iref{i:TSidemspec}$,
$\iref{i:TSuniisot}=\iref{i:TSunispec}$ and
$\iref{i:TSredisot}=\iref{i:TSredspec}$.
\lref{l:semiuniloop} tells us that
$\iref{i:TSuniisom}=\iref{i:TSredisom}$,
$\iref{i:TSuniisot}=\iref{i:TSredisot}$ and
$\iref{i:TSunispec}=\iref{i:TSredspec}$.
Finally, $\iref{i:TSloop}=\iref{i:TSredisom}$ follows from the definition
of a loop, and the fact that any loop has an isomorph in which 1 is the
identity element.
\end{proof}

Using \tref{t:STS} and the known results on enumeration of Steiner
triple systems \cite{KO04}, we immediately have the numbers of
isomorphism classes shown in \Tref{tabSTS} (hence we performed no
computations for this table). Orders below 19 which are not shown in
the table are known not to have any Steiner triple systems.

\begin{table}
\def\qd{\quad}
\newdimen\digitwidth \setbox0=\hbox{\rm0} \digitwidth=\wd0
\catcode`@=\active \def@{\kern\digitwidth} \centerline{
\vbox{\offinterlineskip \hrule 
\halign{&\vrule#&\strut
        \qd\hfil#\hfil\qd&\vrule#&
        \qd\hfil#\hfil\qd&\vrule#&
        \qd\hfil#\hfil\qd&\vrule#\cr
height2pt&\omit&&&&&\cr
&order&&species&&all squares&\cr
height2pt&\omit&&&&&\cr
\noalign{\hrule height0.8pt}
height2pt&\omit&&&&&\cr
&3&&1&&1&\cr
&7&&1&&30&\cr
&9&&1&&840&\cr
&13&&2&&1197504000&\cr
&15&&80&&60281712691200&\cr
&19&&11084874829&&1348410350618155344199680000&\cr
height2pt&\omit&&&&&\cr
}\hrule}}
\caption{\label{tabSTS}Number of totally symmetric idempotent Latin squares.
This table also counts other objects, by \tref{t:STS}.}
\end{table}

We finish by briefly describing the computations which produced
\Tref{tabtotsym}. These were similar to the computations for semisymmetric
Latin squares in \sref{s:semisym}. We first installed the entries on
the main diagonal (and any entries they implied), and
screened for isomorphism.
All other entries come in sets of six: $L[i,j]=L[j,i]=k$,
$L[i,k]=L[k,i]=j$ and $L[j,k]=L[k,j]=i$ for distinct $i,j,k$.  These
were filled in one 6-tuple at a time in backtracking fashion, while
respecting the Latin property. Just as we did for semisymmetric Latin
squares, our two independent computations screened for isomorphism at
different points by canonically labelling and sorting down to
inequivalent subcases.

The various types of equivalence of Latin squares can be tested by
converting the squares to graphs, as described
in~\cite[Theorem~7]{MMM07}, and processed using
\texttt{nauty}~\cite{Nauty}.  With squares having conjugate symmetry,
we can often take advantage of the symmetry to construct a smaller
graph, which allows faster processing.  We illustrate with the example
of testing isomorphism of a totally symmetric square~$L$ of order~$n$.
Define a directed graph $G(L)$ with vertices $V_1\cup V_2$ where
$V_1=\{1,\ldots,n\}$.  There is a directed edge $(u,v)$ for each
triple $(u,u,v)$ with $u\ne v$.  For each $\{u,v,w\}$ such that
$(u,v,w)$ is a triple of $L$ and $|\{u,v,w\}|=3$, there is one vertex in
$V_2$ adjacent to $u$, $v$ and~$w$.  The vertices of $V_1$ are
coloured differently from the vertices in~$V_2$.  It is clear that $L$
can be uniquely reconstructed from~$G(L)$ (there is a triple $(u,u,u)$
for each $u\in V_1$ with no directed edge leaving $u$).  Moreover, totally
symmetric Latin squares $L_1$ and $L_2$ are isomorphic if and only if
$G(L_1)$ and $G(L_2)$ are isomorphic as coloured graphs.  Relabelling
$L$ according to the order induced on $V_1$ by a canonical labelling
of $G(L)$ produces a canonical representative of the isomorphism class
of~$L$.

Since the graphs produced by these constructions tend to be highly
regular, the efficiency of \texttt{nauty} can be enhanced by using
invariants to separate inequivalent vertices.  Two invariants that
proved useful were the cycle structure of the rows (columns, symbols)
relative to other rows (resp. columns, symbols), and the distribution
of intercalates ($2\times 2$ Latin subsquares).

  \let\oldthebibliography=\thebibliography
  \let\endoldthebibliography=\endthebibliography
  \renewenvironment{thebibliography}[1]{%
    \begin{oldthebibliography}{#1}%
      \setlength{\parskip}{0.4ex plus 0.1ex minus 0.1ex}%
      \setlength{\itemsep}{0.4ex plus 0.1ex minus 0.1ex}%
  }%
  {%
    \end{oldthebibliography}%
  }

\subsection*{Acknowledgements}
The authors are very grateful to Petr Vojt\v{e}chovsk\'y and Michael Kinyon
for alerting them to reference \cite{Art59}.

\end{document}